\documentclass[bj,11pt]{imsart}
\setattribute{journal}{name}{}
\usepackage{amssymb,color,enumerate}
\usepackage{graphicx}
\usepackage{pdfsync,url}
\usepackage{amsmath}
\usepackage{wrapfig}
\usepackage{tikz}
\usepackage{amsfonts, amsthm}
\usepackage{hyperref}
\usepackage[square,sort&compress]{natbib}
\setcitestyle{numbers}

\theoremstyle{plain}
 
 \newtheorem{prop}{Proposition}[section]
 \newtheorem{lem}{Lemma}[section]
 
\theoremstyle{definition}
 
 \newtheorem{rem}{Remark}[section]
 \newtheorem{dfn}{Definition}[section]
\numberwithin{equation}{section}
\renewcommand{\le}{\leqslant}\renewcommand{\leq}{\leqslant}
\renewcommand{\ge}{\geqslant}\renewcommand{\geq}{\geqslant}

\newcommand{\R}{\mathbb{R}}

\newcommand{\M}{\mathbb{M}}

\newcommand{\Prob}{\mathbb{P}}
\newcommand{\PR}{\mathbb{P}}
\newcommand{\E}{\mathbb{E}}

\newcommand{\vague}{\stackrel{\lower0.2ex\hbox{$\scriptscriptstyle
                    \it{v} $}}{\rightarrow}}
\newcommand{\weak}{\stackrel{\lower0.2ex\hbox{$\scriptscriptstyle
                    \it{w} $}}{\rightarrow}}
\newcommand{\what}{\stackrel{\lower0.2ex\hbox{$\scriptscriptstyle
                    \it{\hat{w}} $}}{\rightarrow}}
\newcommand{\eqdis}{\stackrel{\lower0.2ex\hbox{$\scriptscriptstyle
                    \mathrm{d}$}}{=}}
\newcommand{\distr}{\stackrel{\lower0.2ex\hbox{$\scriptscriptstyle
                    \it{d} $}}{\rightarrow}}

\def\bX{\boldsymbol X}
\def\bY{\boldsymbol Y}

\def\bZ{\boldsymbol Z}

\def\bC{\mathbb{C}}

\def\bone{\boldsymbol 1}
\def\ba{\boldsymbol a}

\def\bzero{\boldsymbol 0}

\def\bx{\boldsymbol x}
\def\bz{\boldsymbol z}

\def\by{\boldsymbol y}
\def\binfty{\boldsymbol \infty}

\def\btheta{\boldsymbol \theta}
\def\bbeta{\boldsymbol \beta}

\def\polar{\text{POLAR}}
\def\gpolar{\text{GPOLAR}}

\def\diag{\textsf{diag}}
\def\smallwedge{\textsf{wedge}}

\newcommand\independent{\protect\mathpalette{\protect\independenT}{\perp}}
\def\independenT#1#2{\mathrel{\rlap{$#1#2$}\mkern2mu{#1#2}}}

\def\RV{\mathcal{RV}}

\allowdisplaybreaks 
\numberwithin{equation}{section}
\def\Unif{\text{Unif}}
\def\Hillish{\text{Hillish}}
\def\CEV{\text{CEV}}
\def\MRV{\text{MRV}}
\def\HRV{\text{HRV}}
\newcommand{\cinP}{\stackrel{\lower0.2ex\hbox{$\scriptscriptstyle
                    \it{P} $}}{\rightarrow}}

\def\bbeta{\mathbb{\eta}}
\def\bbxi{\mathbb{\xi}}

\definecolor{darkred}{RGB}{200,0,0}
\definecolor{darkblue}{RGB}{0,0,139}

\begin{document}

\begin{frontmatter}

\title{Hidden Regular Variation under Full and Strong Asymptotic
  Dependence}
\runtitle{Strong Asymptotic Dependence}
\author{\fnms{Bikramjit} \snm{Das}}
\address{Engineering Systems and Design,\\ 
Singapore University of Technology and Design,\\
Singapore 487372 \\ {\sc{Email:}} \tt{bikram@sutd.edu.sg} }
\affiliation{Singapore University of Technology and Design}
\and
\author{\fnms{Sidney I.} \snm{Resnick}}
\address{School of ORIE,\\ Cornell University,\\
Ithaca, NY 14853 USA\\{\sc{Email:}} \tt{sir1@cornell.edu} }
\affiliation{Cornell University}

\runauthor{Das and Resnick}

\begin{abstract}
Data exhibiting heavy-tails in one or more dimensions is often
  studied using the {framework} of regular variation. In a multivariate
  setting this {requires identifying} specific forms of dependence
  in the data; this means {identifying that the} data tends to
  concentrate along particular 
  directions and does not cover the full space. This is observed in
  various data sets {from}
 finance, insurance, network traffic, social networks, etc. In
  this paper we discuss the notions of full and
strong asymptotic dependence for bivariate data {along with the idea
  of hidden regular variation in these cases}. 
{In a risk analysis setting, this leads to improved risk estimation
  accuracy  when regular methods provide a zero estimate of risk}. Analyses of 
both real and simulated data sets  illustrate concepts {of generation and detection of such models}.\end{abstract}



\end{frontmatter}

\section{Introduction}

Data that may be modeled by distributions having  heavy tails appear  in many contexts, for example,
 hydrology (\cite{anderson:meerschaert:1998}), finance
  (\cite{smith:2003}), insurance
  (\cite{embrechts:kluppelberg:mikosch:1997}), Internet traffic
  (\cite{crovella:bestavros:taqqu:1998}), social networks and random
  graphs (\cite{durrett:2010, bollobas:borgs:chayes:riordan:2003,
    resnick:samorodnitsky:towsley:davis:willis:wan:2016,
    resnick:samorodnitsky:2015})
  and risk
management (\cite{das:embrechts:fasen:2013,ibragimov:jaffee:walden:2011}).
  Empirical evidence often indicates heavy-tailed marginal 
distributions and the
dependence structure between the various
components must be discerned. We focus here on  the case where
components are strongly dependent.

The purpose of this paper is twofold. First, the paper encourages
 a definition of strong asymptotic {or extremal}
 dependence that means  the limit measure of regular variation
 concentrates on a cone smaller than the full state space. 
Thus,
 directions where multivariate data from such a model {are} found fall in a
 restricted set. Secondly, the paper shows that strong asymptotic dependence 
 is a tractable case for 
the applicability of {\it hidden regular variation\/}.

Hidden regular variation (HRV)
\cite{resnick:2002a,resnick:2008,resnickbook:2007,das:mitra:resnick:2013,
  das:resnick:2015, lindskog:resnick:roy:2014} is often considered for 
 multivariate data exhibiting heavy tails  when {\it
  asymptotic independence\/} is present. Asymptotic
  independence in a bivariate data set of positive values implies that
  both coordinates cannot be large 
  simultaneously and therefore the multivariate regular variation (MRV) limit measure
  concentrates on the co-ordinate axes. To improve risk estimation,
  one then seeks HRV on the 
  non-negative orthant after removing the two axes.

However, hidden regular variation
is applicable whenever the limit measure of regular variation
in standard scale concentrates on a cone which is smaller than the
entire state space, and {is} not restricted to the case of
asymptotic independence; see \cite{das:mitra:resnick:2013, 
  lindskog:resnick:roy:2014} for details. If the limit measure
of regular variation concentrates on a relatively small cone, a risk
calculation of a region in the complement of the support of the limit
measure will yield an answer of {zero} and HRV has the potential to
produce positive estimates of such risks.

We distinguish two related cases:
\begin{enumerate}
\item \emph{Full asymptotic dependence:}  the MRV limit measure  concentrates
  in standard scale on a single diagonal ray.
\item {\emph{Strong asymptotic dependence:}  the MRV} limit measure in standard
  scale concentrates on a relatively small cone about the
  diagonal. This case is illustrated by analyzing out- and in-degree
  for Facebook wall posts 
  and returns of 
 Chevron vs
  Exxon. The variables 
in these examples are
 highly dependent, but they  are not fully asymptotically 
  dependent. In our experience, it is much easier to find examples of strong asymptotic
  dependence compared with full asymptotic dependence.
\end{enumerate}

We {review and adapt}    general  model
 generation
and detection techniques based on the generalized 
polar coordinate transform; see \cite[p. 198]{resnickbook:2007}, \cite{das:mitra:resnick:2013,
  das:resnick:2015, lindskog:resnick:roy:2014}).
The model generation {methods} produce tractable models and the
methods are
illustrated
 in Sections \ref{ex:sim1} and \ref{ex:sim2}. 
The detection methods show when regularly varying models are
consistent with data.
We apply the detection methods to the data examples of strong asymptotic dependence.

The mathematical framework for the study of multivariate heavy tails
is regular variation of measures.
The theory is flexible when given for closed subcones of metric spaces
\cite{lindskog:resnick:roy:2014};   we specialize to subcones of 
 $\mathbb{R}_+^2$ and $\mathbb{R}^{2}$ where statistical results are most readily
exhibited. Statistical extensions to higher dimensions are possible
{and require more sophisticated graphics.}
We list needed
notation  in Section \ref{subsec:notation} for reference. The
definitions of multivariate regular variation (MRV) and hidden regular
variation (HRV) are reviewed in Section \ref{subsec:regVarMod} where general concepts
  are adapted for subcones in two
dimensions.  Sections
  \ref{subsec:polar} and \ref{subsec:MRVHRV} give
equivalent formulations in polar co-ordinates and discuss the
particular cases of strong {and full} dependence.

 Section \ref{subsec:Hillish} gives techniques for
detecting when data is consistent with a model exhibiting MRV and
HRV. These techniques rely on the fact that under 
broad conditions, if  a vector $\bX$ has a  multivariate
 regularly varying distribution on a cone $\mathbb{C}$, then under a
{\it generalized polar coordinate transformation\/} (see
\eqref{eq:defgpolar}), the transformed 
vector satisfies a conditional extreme value (CEV) model for which 
detection techniques exist from \cite{das:resnick:2011b}.
{This methodology adds to the toolbox of }
  one dimensional techniques such as checking if one
  dimensional marginal distributions are heavy tailed
  or checking whether one dimensional functions of the data vector such as the maximum and the  minimum
component are heavy tailed.  See \cite[p. 326]{resnickbook:2007}, \cite{resnick:2002a}.
In Section \ref{sec:dataAnal}, we analyze real and simulated data and
show that our estimation and detection techniques {produce results
consistent with presence of  both MRV with strong dependence and HRV.}
   Section \ref{sec:conc} presents  concluding
  comments.

\section{{Background on Regular Variation of Measures}}\label{sec:basics}
We provide a brief review of the mathematical setup for multivariate regularly
varying measures with the notion of $\M$-convergence. More detail is found in \cite{hult:lindskog:2006a,
hult:lindskog:2006b,das:mitra:resnick:2013,das:resnick:2015,lindskog:resnick:roy:2014}. The
notions of hidden regular variation (HRV) and regular variation
expressed by polar coordinate transforms are discussed in Sections
\ref{subsec:regVarMod} and \ref{subsec:polar} with  emphasis on cases of
the \emph{strong asymptotic dependence}. Finally in Section
\ref{subsec:Hillish}  we discuss detection
of HRV using the Hillish estimator.
\subsection{Basic notation.}\label{subsec:notation}
A summary of some notation and concepts are provided here. {For this paper, we restrict to dimension $d=2$ unless otherwise specified.}  We use bold
  letters to denote vectors, with capital letters for random vectors
  and small letters for non-random vectors, e.g.,
  $\by=(y_1,y_2)\in \R^2$. We also define
  $\bzero=(0,0)$ and $\binfty=(\infty,\infty)$. Vector operations are always understood component-wise,
  e.g., for vectors $\bx$ and $\by$, $\bx\le \by$ means $x_i\le y_i$
  for $i=1, 2$. Some additional notation follows with
  explanations that  are amplified in subsequent sections. 
Detailed discussions are in the references.
$$ 
\begin{array}{llll}
\RV_\beta & \text{Regularly varying functions with index $\beta>0$; that
           is, functions $f:\mathbb{R}_+\mapsto \mathbb{R}_+$}\\
{}& \text{satisfying $\lim_{t\to\infty}f(tx)/f(t)=x^\beta,$ for $x>0.$  We
  can and do assume such functions }\\
&\text{are continuous and  strictly increasing. See \cite{bingham:goldie:teugels:1989, resnickbook:2008, dehaan:ferreira:2006}.}\\[2mm]
\E & \mathbb{R}_+^2 \setminus \{\bzero\} \text{ or }\mathbb{R}^2 \setminus \{\bzero\}.\\[2mm]
[\diag] & \{(x,x):x\geq 0\}.\\[2mm]
[\smallwedge] & {\{\bx\in\R_+^2: a_lx_1 \leq x_2 \leq a_u x_1\} }\;\;\text{for
                some $0<a_l<a_u<\infty$.}\\[2mm]

\nu_\alpha (\cdot) & \text{The Pareto measure on $(0,\infty)$ given by $\nu_\alpha(x,\infty)=x^{-\alpha},x>0$.}\\[2mm]
\MRV & \text{Multivariate regular variation; for this paper, it means regular variation
  on $\E$}.\\[2mm]
\HRV & \text{Hidden regular variation; for this paper, it means a
       second regular variation after}\\
{}& \text{removal of a cone as well as $\bzero$.}\\[2mm]

\M(\mathbb{C}\setminus \mathbb{C}_0) & \text{The set of all non-zero
measures on $\mathbb{C}\setminus \mathbb{C}_0$ which are finite on
                                       subsets bounded}\\
&\text{away from the \emph{forbidden zone} $\mathbb{C}_0$, {a closed cone
  removed from the state space.}}\\[2mm]
\mathcal{C}(\mathbb{C}\setminus \mathbb{C}_0) & \text{Continuous, bounded,
  positive functions on $\mathbb{C}\setminus \mathbb{C}_0$ whose
  supports are bounded}\\&\text{away from the \emph{forbidden zone}
  $\mathbb{C}_0$. Without 
  loss of generality (\cite{lindskog:resnick:roy:2014}), we
  may assume}\\&\text{the  functions are
  uniformly continuous.}\\[2mm] 
\mu_n\to \mu & \text{Convergence in $\M(\mathbb{C}\setminus
  \mathbb{C}_0)$ means $\mu_n (f) \to \mu(f)$ for all } f \in \mathcal{C}(\mathbb{C}\setminus \mathbb{C}_0). \text{ See
   \cite{hult:lindskog:2006a, das:mitra:resnick:2013,
  lindskog:resnick:roy:2014}}\\
{}& \text{and Definition \ref{dfn:mconv}.}\\[2mm]
d(\bx,\by) & \text{Metric in $\R^2$, usually the $L_2$
             distance  $d(\bx,\by)=\bigl((x_1-y_1)^2+(x_2-y_2)^2\bigr)^{1/2}.$}\\[2mm] 
\text{diamond plot} & \text{Mapping of thresholded data onto the
                      $L_1$ unit sphere }
\bx \mapsto \bigl(\frac{x_1}{|x_1|+|x_2|},\frac{x_2}{|x_1|+|x_2|}\bigr)  \\[2mm]
d(\bx,\mathbb{C}) & \inf\limits_{\by\in\mathbb{C}} d(\bx,\by) \text{ for } \bx \in \E \text{ and } \bC\subset \E.\\[2.5mm]
\aleph_{\mathbb{C}} & \{\bx : d(\bx,\mathbb{C} )=1\}. \text{ For
                      instance: }
\aleph_{\bzero} = \{\bx \in \E : d(\bx, \{\bzero \})=1\}, \\&{}
\aleph_{[\diag]} = \{\bx \in \E: d(\bx,[\diag])=1\} \text{ and }
\aleph_{[\smallwedge] } = \{\bx \in \E: d(\bx,[\smallwedge])=1\}. \\[2mm]
   \end{array}$$
$$\begin{array}{llll}  
\gpolar & \text{{Generalized} polar co-ordinate transformation relative to the
  deleted forbidden}\\
&\text{zone $\mathbb{C}_0$.} \quad\gpolar(\bx)=\left(d(\bx, \mathbb{C}_0),
\bx/d(\bx,\mathbb{C}_0)\right). \text{ See
  \cite{lindskog:resnick:roy:2014, das:mitra:resnick:2013}.}\\[2mm]
\bX \independent \bY & \text{The random elements $\bX$ and  $\bY$ are
                       independent.}
\end{array}
$$

\subsection{Regularly varying distributions on cones.}\label{subsec:regVarMod}
We review material from \cite{hult:lindskog:2006a, das:mitra:resnick:2013,
  lindskog:resnick:roy:2014} describing 
 MRV and HRV specialized to two dimensions.
The convergence concept used for defining regular variation is
  $\M$-convergence which is slightly different from  vague
  convergence  traditionally used. 
Reasons for
  preferring $\M$-convergence are discussed in  \cite{lindskog:resnick:roy:2014, das:mitra:resnick:2013}.

\subsubsection{Forbidden zones.}\label{subsub:forbidden}
Consider $\mathbb{R}_+^2 $ {or $\mathbb{R}^2 $} as a metric space with Euclidean
 metric $d(\bx,\by)$. A subset  $\mathbb{C}$ is a {\emph{cone}} if it is closed under
positive scalar multiplication: if $\bx \in
\mathbb{C}$ then $c\,\bx \in \mathbb{C}$ for $c>0$.
 A framework for discussing regularly varying measures
is $\M$-convergence
(\cite{lindskog:resnick:roy:2014, das:mitra:resnick:2013}) on a closed
cone $\mathbb{C}\subset \mathbb{R}_+^2 $ {or $\R^{2}$} with a closed cone 
$\mathbb{C}_0 \subset  \mathbb{C}$ deleted.
Call the deleted cone 
$\mathbb{C}_0 $ the {\it forbidden zone\/}.

{Here are some cases of interest for this paper; see  Figure
  \ref{fig:diag_wedge}.}

\begin{enumerate}
\item Suppose $\mathbb{C} = \mathbb{R}_+^2$
 and $\mathbb{C}_0 =\{\bzero\}$. Then $\E:=\mathbb{C} \setminus
 \mathbb{C}_0
= \mathbb{R}_+^2 \setminus \{\bzero\}$ is the space for defining
$\M$-convergence appropriate for regular variation of distributions of
positive random vectors. The forbidden zone is the origin
$\{\bzero\}.$

\item  Suppose $\mathbb{C} = \mathbb{R}^2$
 and $\mathbb{C}_0 =\{\bzero\}$. Then $\E:=\mathbb{C} \setminus
 \mathbb{C}_0
= \mathbb{R}^2 \setminus \{\bzero\}$ is the space 
 appropriate for regular variation of distributions of pairs of real valued random
 variables such as those representing financial returns.
The forbidden zone is still the origin
$\{\bzero\}.$

\item Suppose $\mathbb{C} = \mathbb{R}_+^2$
 and $\mathbb{C}_0 =\{(x,x):x\geq 0\}=:[\diag]$. Then
 $\mathbb{C} \setminus 
 \mathbb{C}_0= \R_+^2 \setminus [\diag]$, 
the first quadrant without its diagonal,  is the right space for
 defining $\M$-convergence appropriate for HRV when asymptotic full
 dependence is present. The forbidden zone  is  the diagonal; see Figure \ref{fig:diag_wedge}. 

\item A related example is $\mathbb{C}=\R^{2}$ and $\mathbb{C}_0 =\{(x,x):x\in \R\}$ and we seek regular variation on $\mathbb{C} \setminus 
 \mathbb{C}_0= \R^{2} \setminus \{(x,x):x\in \R\}$. 

\item Suppose $\mathbb{C} = \mathbb{R}_+^2$ and for $0<\theta_l<\theta_u<1$,
  and 
\begin{equation}\label{e:parCone}
0<a_l=\theta_u^{-1}-1 <a_u=\theta_l^{-1}-1 <\infty,
\end{equation}
\begin{align}
\mathbb{C}_0 =&\{\bx\in
  \R_+^2: 0<\theta_l \leq \frac{x_1}{x_1+x_2} \leq \theta_u<1\}
\label{eq:polar} \\
=& \{ \bx \in \R_+^2: a_lx_1 \leq x_2 \leq a_u x_1\}=:[\smallwedge],\label{eq:peewee}
\end{align}
where $[\smallwedge]$ 
is a pizza slice removed from the first quadrant.
 We then seek hidden regular variation  on $\R_+^2\setminus
[\smallwedge] $; see Figure \ref{fig:diag_wedge}. 
 When $a_l=a_u=1$ (or equivalently
$\theta_l=\theta_u=1/2$), then $[\smallwedge]$ reduces to
$[\diag]$. 
Note we {may} parameterize $[\smallwedge ]$ in two ways, {one using
the slopes $a_l,a_u$ and one using angles $\theta_l,\theta_u$}. In
\eqref{eq:polar} we use the traditional $L_1$ polar coordinate
transform $\polar $ from $\R^2\setminus \{\bzero\}\mapsto (0,\infty)\times
[-1,1]$ given by
$$\polar: \bx \mapsto \Bigl(|x_1|+|x_2|, \frac{x_1 }{|x_1|+|x_2 |}\Bigr)=(r,\theta).$$
to express $[\smallwedge]$ in polar coordinates as
$\R_{+}\times[\theta_l,\theta_u]$.  In \eqref{eq:peewee} we 
use the slopes $a_l,a_u$ of the boundary lines of $[\smallwedge]$. \
{In
practice we try to infer $[\smallwedge]$ by making a diamond plot of 
the data using $L_1$ norm thresholding.}
\end{enumerate}

{In this paper we give particular attention to $[\smallwedge ]$
because
\begin{itemize}
\item when the limit measure of regular variation concentrates on
  $[\smallwedge ]\subsetneq \R_+^2$ we have a tractable notion of {\it strong
    asymptotic dependence\/}; and
\item data examples of strong asymptotic dependence seem to be far more common
  than for the case of  full asymptotic dependence.
\end{itemize} 
Of course,
other types of forbidden zones are possible and to date most attention
has been directed to removing axes when asymptotic independence is present.}

\subsubsection{Regular variation of measures.}
Let $\M(\mathbb{C}
\setminus \mathbb{C}_0)$ be the set of Borel measures on $\mathbb{C}
\setminus \mathbb{C}_0$ which are finite on sets bounded away from the
forbidden zone 
$\mathbb{C}_0$ (\cite{das:mitra:resnick:2013,
  hult:lindskog:2006a,lindskog:resnick:roy:2014}).  We think of sets bounded
  away from the forbidden zone $\mathbb{C}_0$ as {\it tail regions\/}.
$\M$-convergence  is the basis for the
 definition of multivariate regular variation:
  \begin{dfn}\label{dfn:mconv}
  For $\mu_n, \mu \in \M(\mathbb{C} \setminus \mathbb{C}_0)$ we say
  $\mu_n \to \mu$ in $\M(\mathbb{C} \setminus \mathbb{C}_0)$   if
  $\int f\mathrm d \mu_n \to \int f \mathrm d \mu$ for all  $f\in \mathcal{C}( \mathbb{C} \setminus
  \mathbb{C}_0)$.
  \end{dfn} 
\begin{dfn}\label{dfn:regvarwithmconv}
A random vector $\bZ\geq \bzero$ is regularly varying on $\mathbb{C}
\setminus \mathbb{C}_0$ with index $\alpha>0$ if there exists 
$b(t) \in \RV_{1/\alpha}$, called the {\it scaling
  function\/}, and a measure $\nu(\cdot) 
\in \M(\mathbb{C} 
\setminus \mathbb{C}_0)$, called the {\it limit or tail measure\/},
{such that} as $t \to\infty$,
\begin{equation}\label{eq:RegVarMeas}
t\,\Prob[ \bZ/b(t) \in \cdot \,] \to \nu(\cdot), \quad \text{ in } \M(\mathbb{C}
\setminus \mathbb{C}_0).
\end{equation}
\end{dfn}
We write $\bZ \in \MRV (\alpha, b(t), \nu, \mathbb{C}
\setminus \mathbb{C}_0)$ to emphasize that regular variation depends
on an 
index $\alpha$, scaling function $b \in \RV_{1/\alpha}$, limit measure $\nu$, and state
space $ \mathbb{C}
\setminus \mathbb{C}_0.$
Since  $b(t) \in \RV_{1/\alpha}$,
the limit measure $\nu(\cdot) $ has a scaling property,
\begin{equation}\label{eq:limMeasScales}
\nu(c \,\cdot) =c^{-\alpha} \nu(\cdot),\qquad c>0.
\end{equation}

Suppose $\mathbb{C}=\mathbb{R}_+^2$, $\mathbb{C}_0=\{\bzero\}$. We distinguish between different forms of dependence and identify them as follows:
\begin{enumerate}
 
\item If  $\nu(\cdot)$ satisfies 
$\nu((0,\infty)^2)=0$ so that
$\nu$ concentrates on the axes, then $\bZ$ possesses {\it asymptotic
  independence\/}; see \cite{resnickbook:2007, dehaan:ferreira:2006,
  resnickbook:2008}.
  \item\label{item:2}  If $\nu(\cdot)$ concentrates on $[\diag]$ then
$\bZ$ has {\it  full asymptotic dependence\/}. 
\item\label{item:3} If $\nu(\cdot)$ concentrates on
a narrow wedge as in \eqref{eq:peewee}, then $\bZ$ has {\it strong asymptotic
 dependence\/}.
\end{enumerate}

An analogous classification can be made for the case
  $\mathbb{C}=\mathbb{R}^2.$ 

{\bf Diamond plot:} 
When { doing empirical analyses for cases \ref{item:2} or \ref{item:3}, it
  is convenient and  informative to map points 
$$ \bx \mapsto
\Bigl(\frac{x_1}{|x_1|+|x_2},\frac{x_2}{|x_1|+|x_2|}\Bigr)
=\btheta =(\theta_1,\theta_2)$$
onto the $L_1$ unit sphere or diamond, perhaps after thresholding data
according to {the} $L_1$ norm. We call the resulting plot the {\it diamond
  plot\/}. Observing how points cluster on the $L_1$ unit sphere provides
a visualization of dependence.}

\subsection{Regular variation and the polar coordinate
  transformation.}\label{subsec:polar}
When the forbidden zone is the origin, it is useful to rephrase
regular variation of measures using the polar coordinate
transformation. Theoretically, we may choose any norm $\|\cdot\|$ and
the polar coordinate transform 
maps $\bx$ into the unit sphere determined by the chosen norm:
$\bx \mapsto \bigl(\|\bx\|, {\bx}/{\|\bx\|}\bigr)$.
The limit measure expressed in polar coordinates is a product measure
and this provides a way to construct regularly varying measures and is
useful for inference. (See \cite[p. 168 ff, 173 ff]{resnickbook:2007}.)
When the forbidden zone is a more general cone than just the
  origin, the polar coordinate transform no longer brings benefits and
  the limit angular measure expressed in these coordinates may be infinite.
To get a limit measure expressed as a product {where the analogue of
the angular measure is a probability measure,} one may
transform \eqref{eq:RegVarMeas} and  
\eqref{eq:limMeasScales} using {\it generalized\/} polar coordinates
(\cite{lindskog:resnick:roy:2014, das:mitra:resnick:2013}).
 We can define
the generalized polar coordinate transform for general cones of the form $\mathbb{C}\setminus\mathbb{C}_0$
and an {associated metric $d(\cdot,\cdot)$ satisfying
$d(c\bx,c\by)=cd(\bx,\by)$ for scalars $c>0$.  The metric $d(\cdot,
\cdot)$ that we use in
practice is the usual $L_2$ Euclidean metric {but note that  the $L_1$ norm is 
used for} visualizations using the diamond plot.}  

\subsubsection{Generalized polar coordinates.}
 Define $\gpolar: \mathbb{C}\setminus \mathbb{C}_0
\mapsto (0,\infty)\times 
\aleph_{\mathbb{C}_0}$ by
\begin{equation}\label{eq:defgpolar}
\gpolar (\bx) =\left(d(\bx,\mathbb{C}_0) , \frac{\bx}
{d(\bx,\mathbb{C}_0)}\right).
\end{equation}
Consequently,  the inverse
  $\gpolar^{\leftarrow}:(0,\infty)\times \aleph_{{\bC}_0}\mapsto 
  \bC\setminus\bC_{0}$ 
of the $\gpolar$ function is 
\begin{equation}\label{eq:defgpolarinv}
\gpolar^{\leftarrow} (r, \theta) =r\theta.
\end{equation}

The transformation $\gpolar$ depends on the
  forbidden zone $\bC_0$ and {the choice of metric}. 
In practice, the metric $d(\cdot,\cdot)$ is taken to be the usual
$L_2$ Euclidean distance.  This is practical and customary but not obligatory.

\subsubsection{Generalized unit sphere.} 
{When transforming from Cartesian to polar coordinates, a central role
is played by the unit sphere $\aleph_0:=\{\bx \neq \bzero:\|\bx\|=1\}.$ 
The comparable set when using generalized polar coordinates with
respect to the forbidden zone $\mathbb{C}_0$   is }
$\aleph_{\mathbb{C}_0}=\{\bx \in \mathbb{C} \setminus
\mathbb{C}_0: d(\bx, \mathbb{C}_0)=1\},$
 the locus of points at distance 1 from the deleted forbidden zone
$\mathbb{C}_0$.
We then have an equivalent form of   \eqref{eq:RegVarMeas} and
\eqref{eq:limMeasScales},
namely,
\begin{equation}\label{eq:limMeasPolar}
t\Prob \Bigl[\gpolar \left(\frac{\bZ}{b(t)}\right) \in \;\cdot
  \;\Bigr] =
t\Prob \Bigl[ \Bigl( \frac{d(\bZ, \bC_{0})}{b(t)},
\frac{\bZ}{d(\bZ,\bC_0)}\Bigr) \in \cdot \,\Bigr]
\to (\nu_{\alpha}  \times
S_0)(\cdot) \,  = (\nu\circ\gpolar^{\leftarrow})(\cdot)\,,
\end{equation}
in $\M\left((0,\infty) \times \aleph_{\mathbb{C}_0} \right)$
where $\nu_{\alpha} (x, \infty) =x^{-\alpha}, \,x>0,\,\alpha>0$
and $S_0(\cdot) $ is a probability measure on $\aleph_{\mathbb{C}_0}
$  (\cite{das:mitra:resnick:2013, lindskog:resnick:roy:2014}), {provided
$b(t)$ is appropriately chosen}. 
{Note that $\aleph_{\mathbb{C}_0}$ depends on the choice of
  $d(\cdot,\cdot)$, and 
the limit in \eqref{eq:limMeasPolar}
is a product measure.}

\subsubsection{Examples of unit spheres.}
Figure \ref{fig:diag_wedge} shows different shapes of
$\aleph_{{\bC_{0}}}$ {for $L_2$ distance} and different choices of $\bC_{{0}}$ where
$\bC= \R_{+}^{2}$.

  \begin{figure}[t]
\centering
\includegraphics[width=6in]{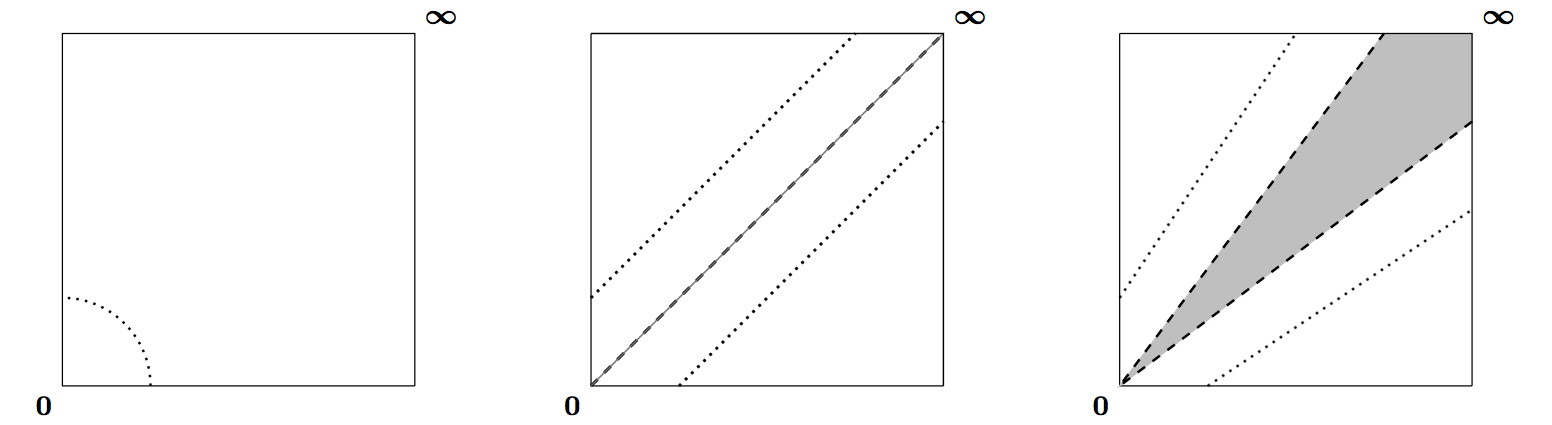}
\vspace{-10pt}
\caption{Left: $\R_{+}^2 \setminus \{\bzero\}$ and distance is $L_2$. Center:  $\R_{+}^2
  \setminus [{\diag}]$ and distance is $L_2$.  Right:  $\R_{+}^2 \setminus
  [\smallwedge]$ and distance is $L_2$. The dotted lines represent $\aleph_{\mathbb{C}_0}$
  which is the locus of points at distance one from $\bC_0$.} 
\label{fig:diag_wedge}
\end{figure}

  \begin{enumerate}[(i)]
\item For $\mathbb{R}_{{+}}^2 \setminus \{\bzero\}$, where the forbidden
  zone is $\{\bzero\}$, we have 
$\aleph_{\bzero}=\{\bx \in \R_{{+}}^2: d(\bx, \{\bzero\})=1\}$.
\item 
If we delete the forbidden zone $\mathbb{C}_0=[\diag]=\{(x,x):x \in \mathbb{R}_{{+}}\}$
from $\R_{{+}}^2$, the  appropriate unit sphere with respect to
$L_2$ distance is 
\begin{equation}\label{eq:alephdiag}
\aleph_{[\diag]}:=\aleph_{>[\diag]} \cup \aleph_{<[\diag]} 
=\{(u,u+\sqrt 2):{u \geq 0} \} \bigcup \{(u,u-\sqrt 2):{u \geq \sqrt 2}\} ,
\end{equation}
the lines of slope 1, above and below the diagonal, which are at
distance 1 from the diagonal. For $ \bx \notin [\diag]$,
\begin{equation}\label{eq:dist(diag)}
d\bigl(\bx , [\diag]\bigr)=|x_1-x_2|/\sqrt{2}.
\end{equation}

\item When the forbidden zone is $[\smallwedge]\subset \R_+^2$, we have
\begin{align}
\aleph_{[\smallwedge]}=&\aleph_{>[\smallwedge]}\cup
\aleph_{<[\smallwedge]}\nonumber \\
=&\{(u,a_uu+\sqrt{1+a_u^2}):u\geq 0\}
\bigcup \{(u,a_lu-\sqrt{1+a_l^2}):u\geq
   a_l^{-1}\sqrt{1+a_l^2}\},\label{eq:alephsmallwedge} 
\end{align}
which are the lines parallel to the two rays defining
$[\smallwedge]$ {at a distance of 1 from $[\smallwedge]$}. When $a_l=a_u=1$, $\aleph_{[\smallwedge]}$ reduces to $\aleph_{[\diag]}$.
For the distance to the forbidden zone from a point $\bx \notin
[\smallwedge]$,  we have 
\begin{align}
d\bigl(\bx, [\smallwedge]\bigr)=& \frac{|x_2-a_ux_1|}{\sqrt{1+a_u^2}}, 
&& \text{ if } x_2>a_ux_1,\label{eq:dist(>wedge)}\\
d\bigl(\bx, [\smallwedge]\bigr)=& \frac{|x_2-a_lx_1|}{\sqrt{1+a_l^2}}, 
&& \text{ if }x_2<a_lx_1,\label{eq:dist(<wedge)}
\end{align}
which reduces to \eqref{eq:dist(diag)} if $a_l=a_u=1$. 
\end{enumerate}

Obvious
  changes apply when the state space is $\R^2$.

\section{MRV and HRV under strong asymptotic dependence.} \label{subsec:MRVHRV}
\subsection{Definitions.}\label{subsub:gen}
 {Consider simultaneous existence of
  regular variation on both
 the big cone $\R_+^2\setminus \{\bzero\}$ and a smaller cone
 $\R_+^2\setminus \mathbb{C}_0$, where $\mathbb{C}_0$ is either
 $[\diag]$ or $[\smallwedge]$.
We  provide equivalent
polar-coordinate  conditions for this simultaneous 
existence. 

\begin{dfn}\label{dfn:MRVHRV}
The vector $
\bZ$ is regularly varying on $
\R_+^2\setminus \{\bzero\}$ and has {\it hidden regular variation\/}  on $\R_+^2\setminus \mathbb{C}_0$ if there exist
$0<\alpha \leq \alpha_0$, scaling functions $b(t) \in
\RV_{1/\alpha}$ and $b_0(t)  \in \RV_{1/\alpha_0}$ with $b(t)/b_0(t) \to
\infty$ and limit measures $\nu,\, \nu_0$ such that 
\begin{equation}\label{eq:mrv+hrv}
\bZ\in
\MRV(\alpha, b(t), \nu, \R_+^2\setminus \{\bzero\}) \cap
\MRV(\alpha_0, b_0(t), \nu_0, \R_+^2\setminus \mathbb{C}_0).
\end{equation}
\end{dfn}
Unpacking the notation we obtain the two regular variation limits
\begin{align}
t\Prob[\bZ / b(t) \in  \cdot \,]\to & \,\nu (\cdot) &&\quad \text{ in
}\M(\R_+^2\setminus \{\bzero\}),
\label{eq:regVarE}\\
t\Prob[\bZ / b_0(t) \in
 \cdot \,]\to & \,\nu_0 (\cdot) &&\quad \text{ in
} \M(\R_+^2\setminus \mathbb{C}_0).\label{eq:regVarE0}
\intertext{
Using
polar coordinates, \eqref{eq:regVarE} can be written as}
t\Prob\left[\left(\|\bZ \|/ b(t), \bZ/\|\bZ\|\right) \in \cdot
  \,\right]\to & \,
\nu_\alpha  \times S (\cdot) &&\quad \text{ in } \M((0,\infty)\times
\aleph_{\bzero  }), \label{eq:regVarEPolar}
\end{align}
where $S$ is a probability measure on $\aleph_{\bzero}$. 
{Similarly when removing $\mathbb{C}_0$ from the state space, 
generalized polar coordinates allow re-writing}
\eqref{eq:regVarE0}  as
\begin{equation}\label{e:delC0}
t\Prob \Bigl[ \Bigl( \frac{d(\bZ, \bC_{0})}{b_0(t)},
\frac{\bZ}{d(\bZ,\bC_0)}\Bigr) \in \cdot \,\Bigr]
\to (\nu_{\alpha_{0}}  \times
S_0)(\cdot) 
\end{equation}
in $\M\left((0,\infty) \times \aleph_{\mathbb{C}_0} \right)$
where $\nu_{\alpha_{0}} (x, \infty) =x^{-\alpha_{0}}, \,x>0,\,\alpha_{0}>0$
and $S_0(\cdot) $ is a probability measure on $\aleph_{\mathbb{C}_0}$.

\subsection{Regular variation when deleting $[\smallwedge]$.}
{Focus on the special case where the forbidden zone is
$[\smallwedge]$. 
Since $[\diag ]$ is a particular case of
$[\smallwedge]$, we do not treat $[\diag]$ separately.
Recall the notation in \eqref{eq:alephsmallwedge} and the two
parameterizations of $[\smallwedge ]$ given in \eqref{e:parCone},
\eqref{eq:polar} and \eqref{eq:peewee}. {The} distance of points to
$[\smallwedge]$ is given in \eqref{eq:dist(>wedge)} and \eqref{eq:dist(<wedge)}.}
When  $\mathbb{C}_0=[\smallwedge]$,
\eqref{e:delC0}  becomes two statements. With $x>0$ and
$\Lambda \subset
  \aleph_{[\smallwedge]}$ we have
\begin{align}
t\Prob \left[ \frac{Z_2-a_uZ_1}{{b_{0}(t)}\sqrt{1+a_u^2}}>x,
                                     \frac{\sqrt{1+a_u^2}\bZ}{Z_2-a_uZ_1}
                                     \in \Lambda \right]
\to & x^{-\alpha_0} S_0(\Lambda) &&\quad \text{ in }
                                   \M\left((0,\infty)\times
                                   \aleph_{>[\smallwedge]}\right), \label{eq:wedge1st}\\
t\Prob \left[ \frac{a_lZ_1-Z_2}{{b_{0}(t)}\sqrt{1+a_l^2}}>x,
                                     \frac{\sqrt{1+a_l^2}\bZ}{a_lZ_1-Z_2}
                                     \in \Lambda \right]
\to & x^{-\alpha_0} S_0(\Lambda) &&\quad \text{ in }
                                   \M\left((0,\infty)\times
                                   \aleph_{<[\smallwedge]}\right) \label{eq:wedge2nd}.
\end{align}
Modifying variables leads to the two simpler statements. For $x>0$,
\begin{align}
t\Prob \left[
\frac{Z_2-a_uZ_1}{b_0(t)}
>x, \frac{Z_2}{Z_1} \leq y\right] \to & (1+a_u^{2})^{\frac{\alpha_0}{2}} x^{-\alpha_0}
                                        S_0 \Biggl\{\Bigl(v,a_uv+\sqrt{1+a_u^2}\Bigr):
                                        v\geq
                                        \frac{\sqrt{1+a_u^2}}{y-a_u}\Biggr\},\;
                                                                        y>a_u,
\label{eq:>smallwedge}\\
t\Prob \left[
\frac{a_lZ_1-Z_2}{b_0(t)}
>x, \frac{Z_1}{Z_2} \leq y\right] \to & (1+a_l^{2})^{\frac{\alpha_0}{2}} x^{-\alpha_0}
                                        S_0\Bigl\{ \bigl( v+\sqrt{1+a_l^2}/a_l,v\bigr):
                                        v\geq
                                        \frac{\sqrt{1+a_l^2}}{a_ly-1}\Bigr\},\;
                                                                        y>\frac{1}{a_l}.
\label{eq:<smallwedge}
\end{align}


\begin{rem}\label{rem:genius}
{\rm Thus a necessary condition for non-trivial regular variation on $\R_+^2\setminus
[\smallwedge]$ is that both $(Z_2-a_uZ_1)_+$ and $(a_lZ_1-Z_2)_+$ be
regularly varying with index $\alpha_0 \geq \alpha$. This fact
suggests the exploratory diagnostic of testing whether these 1
dimensional variables have power laws with the same index. If so, one
can continue to explore with the Hillish statistic and associated
plot; see Section \ref{subsec:Hillish}.

However, there is nothing to prevent the possibility that regular variation
exists on the region above $[\smallwedge]$ but that tails are of lower
order  below $[\smallwedge]$. This would happen for instance if
 $S_{0}(\aleph_{<[\smallwedge]}) =0$ but
 $S_{0}(\aleph_{>[\smallwedge]})>0$.  If this happened, one could
 search for another higher index or thinner tailed 
 regular variation on $<[\smallwedge].$  }
\end{rem}

\begin{rem}\label{rem:altgenius}
{{\rm Analogous statements to 
the limits \eqref{eq:>smallwedge}, \eqref{eq:<smallwedge}
hold true
when $\bC=\R^{2}$ and $[\diag]$ and $[\smallwedge]$ are their appropriate equivalents in $\R^{2}$ to get.}
}\end{rem}

{
\subsubsection{Restrictions on the choice of $[\smallwedge]$.}
 In this paper, we assume that $1< a_{l}\le 1\le a_{u}<\infty$ where $[\smallwedge] = \{\bx\in \R_{+}^{2}: a_{l}x_{1}\le x_{2}\le a_{l}x_{1}\}$. This choice is partly governed by the fact that it is easier for us to deal with data portraying tail equivalence with  $ \lim_{t\to\infty} {\PR(Z_{1}>t)}/{\PR(Z_{2}>t)}=1.$ Now, when $a_{l}=a_{u}$, we get $[\smallwedge]=[\diag]$, which means
 under a model of full asymptotic dependence the only limit measure that we allow is restricted to $[\diag]$.  If we assume
$ \bZ\in \MRV(\alpha, b(t), \nu, \R_+^2\setminus \{\bzero\}) $ and $\nu$ is supported on $[\diag]$
then using \eqref{eq:regVarE}, this clearly implies, that
 \begin{align*}
 \lim_{t\to\infty} \frac{\PR(Z_{1}>t)}{\PR(Z_{2}>t)} & = \lim_{t\to\infty} \frac{\PR(Z_{1}>b(t))}{\PR(Z_{2}>b(t))}\\
     & = \lim_{t\to\infty} \frac{t\,\PR\left(\bZ/b(t) \in (1,\infty)\times[0,\infty)\right)}{t\,\PR\left(\bZ/b(t) \in [0,\infty)\times(1,\infty)\right)}\\
   & = \frac{\nu\left( (1,\infty)\times[0,\infty)\right)}{\nu\left( [0,\infty)\times(1,\infty)\right)}\\
   & =  \frac{\nu\left( (x_{1},x_{2}): x_{1}>1, (x_{1},x_{2})\in [\diag]\right)}{\nu\left( (x_{1},x_{2}): x_{2}>1, (x_{1},x_{2})\in [\diag]\right)} =1,
 \end{align*}
 where the last line is a consequence of $\nu$ being concentrated on $[\diag]$.  Thus,  not only are $Z_{1}, Z_{2}$ tail equivalent, but in fact
 $ \lim_{t\to\infty} {\PR(Z_{1}>t)}/{\PR(Z_{2}>t)}=1.$ 
 
 The following lemma shows that that we cannot choose a $[\smallwedge]$
 which does not contain $[\diag]$ if we want to guarantee  $
 \lim_{t\to\infty} {\PR(Z_{1}>t)}/{\PR(Z_{2}>t)}=1.$  
 
  \begin{lem}
 If $\bZ\in \MRV(\alpha, b(t), \nu, \R_+^2\setminus \{\bzero\})$ where $\nu$ is supported on $[\smallwedge]=\{\bx\in \R_{+}^{2}: a_{l}x_{1}\le x_{2}\le a_{l}x_{1}\}$
 and  
 $ \lim_{t\to\infty} {\PR(Z_{1}>t)}/{\PR(Z_{2}>t)}=1$ then $0<a_{l} \le 1\le a_{u}<\infty$.
 \end{lem}
\begin{proof}
{To get a contradiction,} suppose we have  $1<a_{l}\le
a_{u}<\infty.$
{(A similar contradiction is obtained if we assume 
 $0<a_{l}\le a_{u}<1.$) } Assume,
\begin{align*}
1= \ \lim_{t\to\infty} \frac{\PR(Z_{1}>t)}{\PR(Z_{2}>t)} & 
      = \lim_{t\to\infty} \frac{t\,\PR\left(\bZ/b(t) \in (1,\infty)\times[0,\infty)\right)}{t\,\PR\left(\bZ/b(t) \in [0,\infty)\times(1,\infty)\right)}
        = \frac{\nu\left( (1,\infty)\times[0,\infty)\right)}{\nu\left( [0,\infty)\times(1,\infty)\right)}.
\end{align*}
Hence $\nu\left( (1,\infty)\times[0,\infty)\right)=\nu\left( [0,\infty)\times(1,\infty)\right)$. We know that $\nu$ is supported on $[\smallwedge]$.
So,
\begin{align*}
\nu\left( (1,\infty)\times[0,\infty)\right) & = \nu ((x_{1},x_{2})\in
                                              \R_{+}^{2}: x_{1}>1,
                                              (x_{1},x_{2}) \in
                                              [\smallwedge]
                                              )&&{}\\ 
                        & = \nu \left((x_{1},x_{2})\in \R_{+}^{2}:
                          x_{1}>1, a_{l}\le \frac{x_{2}}{x_{1}}\le
                          a_{u} \right) &&\\  
                        & < \nu \left((x_{1},x_{2})\in \R_{+}^{2}:
                          x_{1}>{\frac {1}{a_{l}}}, a_{l}\le
                          \frac{x_{2}}{x_{1}}\le a_{u}
                          \right)&&(\text{since $a_l>1$}) \\
                          & < \nu \left((x_{1},x_{2})\in \R_{+}^{2}:
                            x_{2}>1, a_{l}\le \frac{x_{2}}{x_{1}}\le
                            a_{u} \right)&&\text{(since $x_2\geq a_lx_1$)}\\
                          & =\nu\left( [0,\infty)\times(1,\infty)\right),&&{}
\end{align*}
which is a contradiction.
\end{proof}

 }

\subsection{Regular variation when deleting $[\smallwedge]$
  expressed in traditional polar coordinates.}
Regular variation on $\R_+^2 \setminus [\smallwedge]$
expressed using generalized polar coordinates in
\eqref{eq:>smallwedge},
\eqref{eq:<smallwedge} can also be written in terms of the traditional
polar coordinates $\bz\mapsto \bigl(r,(\theta,1-\theta)\bigr)$ where
$z_1=r\theta$ and $z_2=r(1-\theta)$ and $r=z_1+z_2.$ Using capital
letters for random variables, the left most probability in
\eqref{eq:>smallwedge} becomes for 
$x>0,\,y>a_u$,
\begin{align*}
t\Prob \Bigl[
\frac{Z_2-a_uZ_1}{b_0(t)}
>x, &\frac{Z_2}{Z_1} \leq y \Bigr] =
t\Prob \left[ \frac{R(1-\Theta (1+a_u))}{b_0(t)}>x,
                                     \frac{(1-\Theta)}{\Theta} \leq y
                                     \right] \\
\intertext{and using \eqref{e:parCone} this is}
=&t\Prob \left[ \frac{R(1-\theta_l^{-1}\Theta )}{b_0(t)}>x,
  \Theta^{-1} \leq y+1
                                     \right] =t\Prob \left[
   \frac{R(1-\theta_l^{-1}\Theta )}{b_0(t)}>x, \Theta\ge\frac{1}{1+y}\right]
\end{align*}
Set $s=1/(1+y)$ where $y>a_u$ and thus $s<\theta_l$,
$$
t\Prob \left[
   \frac{R(1-\theta_l^{-1}\Theta )}{b_0(t)}>x, \Theta>s\right]
\to (1+a_u^{2})^{\frac{\alpha_0}{2}} x^{-\alpha_0}
                                        S_0 \Biggl\{\Bigl(v,a_uv+\sqrt{1+a_u^2}\Bigr):
                                        v\geq
                                        \frac{\sqrt{1+a_u^2}}{s^{-1}-\theta_l^{-1}}\Biggr\}.
$$                                                                       
An analogous expression holds for \eqref{eq:<smallwedge}. So if
regular variation with index $\alpha$ holds on $\R_+^2\setminus
\{\bzero\}$, $R$ is
a random variable with regularly varying distribution tail with index
$\alpha$ and  multiplying $R$ by $1-\theta_l^{-1}\Theta$ produces a
variable with a lighter tail having index $\alpha_0$.

\subsection{The forbidden zone of HRV and the limit measure on    {$\E$}.}
    
Suppose there are two regular variation properties that hold for a vector $\bZ\ge 0$
so that \eqref{eq:mrv+hrv} holds with $b(t)/b_0(t)\to\infty$. If 
$\mathbb{C}_0=[\smallwedge]$ with {$0< a_l \le 1 \le a_u <\infty$}; then
$\nu$, the limit measure on $\R_+^2\setminus \{\bzero\}$, must concentrate
on the forbidden zone $[\smallwedge]$ used to define the second
regular variation. (Cf. \cite[p. 324-5]{resnickbook:2007}.)  This
means that when detecting MRV, if the limit measure is typical of
 strong asymptotic dependence and concentrates on $[\smallwedge]$, we
are encouraged to look for additional regular variation regimes on
$\R_+^2\setminus [\smallwedge]$. Hence we get the following result.

\begin{prop}\label{prop:forbid}
Suppose
\begin{equation}\label{e:doublewhammy}
\bZ\in
\MRV(\alpha, b(t), \nu, \R_+^2\setminus \{\bzero\}) \cap
\MRV(\alpha_0, b_0(t), \nu_0, \R_+^2\setminus [\smallwedge])
\end{equation}
with $b(t)/b_0(t)\to \infty$ and $0 {<} a_l\leq 1\leq a_u<\infty.$ Then 
$\nu$, the limit measure on $\R_+^2\setminus \{\bzero\}$, concentrates
on $[\smallwedge]$. 
\end{prop}

\begin{proof}
To see
this, consider the region above the ray $y=a_ux$. Then for $\delta>0$,
as $t\to\infty$, referring to \eqref{eq:>smallwedge},
\begin{align*}
t\Prob \Bigl[ \frac{Z_2-a_uZ_1}{b(t)} >\delta \Bigr]
=&t\Prob \Bigl[ \frac{Z_2-a_uZ_1}{b_0(t)} > \frac{b(t)}{b_0(t)}\delta\Bigr]
   \to 0,
\end{align*}
 since $b(t)/b_0(t) \to \infty$ and $S_0(\cdot)$ is a probability
measure. Similarly, for the region below the ray $y=a_lx$,
\begin{align*}
t\Prob \Bigl[ \frac{a_lZ_1 -Z_2}{b(t)} >\delta \Bigr]
=&t\Prob \Bigl[ \frac{a_lZ_1-Z_2}{b_0(t)} > \frac{b(t)}{b_0(t)}\delta \Bigr]
\to 0,
\end{align*}
from \eqref{eq:<smallwedge}. Thus $\nu$ places no mass outside $[\smallwedge]$.
\end{proof}
{Clearly, a result analogous to Proposition \ref{prop:forbid} holds where $\bC=\R^{2}$ and $\bC_{0}$ is the appropriate equivalent of $[\diag]$ and $[\smallwedge]$ on $\R^{2}$.} 

{\subsection{How HRV on $\R^2\setminus [\smallwedge]$ can improve
  risk estimates.}\label{subsec:riskestimate}
Suppose $I_1$ and $I_2$ are financial instruments that have positive
risks $Z_1$ and $Z_2$ per unit of investment where $\bZ=(Z_1,Z_2)$
satisfies \eqref{e:doublewhammy}. Suppose we buy one unit of $I_2$
and sell $2a_l$ units of $I_1$. The risk of this portfolio is
$Z_2-2a_lZ_1$ and we have two asymptotic regimes that can be used to
estimate the probability the risk is large. If we use MRV with scale
function $b(t)$ then for large $x>0$,
\begin{align*}
\Prob [Z_2-2a_u Z_1>x]=&\Prob \Bigl[\Bigl(\frac{Z_1}{b(t)},\frac{Z_2}{b(t)} \Bigr) \in
\{(v,w)\in \R_+^2: w-2a_u v>x/b(t)\} \Bigr]\\
\approx & \frac{1}{t} \nu \{(v,w)\in \R_+^2: w-2a_u v >x/b(t)\} =0\\
\intertext{since the required region is outside the support
  $[\smallwedge]$ of the measure $\nu$. Is the risk really $0$ or did
  we use the wrong asymptotic approximation? If we use HRV with
  scale function $b_0(t)$, then 
we get a non-zero limit:}
\Prob[Z_2-2a_u Z_1>x]=&\Prob \Bigl[\Bigl(\frac{Z_1}{b_0(t)},\frac{Z_2}{b_0(t)} \Bigr) \in
\{(v,w)\in \R_+^2: w-2a_u v>x/b_0(t)\} \Bigr]\\
\approx & \frac{1}{t} \nu_0 \{(v,w)\in \R_+^2: w-2a_u v >x/b_0(t)\}.\\
\intertext{Now switch to generalized polar coordinates with
  $(v,w)=r(\mu_1,\mu_2) $ and $\mu_2=a_u\mu_1$ and the risk
  calculation is with respect to the product measure $\nu_{\alpha_0}
  \times S_0\bigl(d(\mu_1,\mu_2)\bigr)$ and }
\Prob[Z_2-2a_u Z_1>x]\approx &\frac 1t \iint_{\{(r,(\mu_1,\mu_2)):
                        r\mu_2-2a_u r\mu_1>x/b_0(t)\}} \alpha_0
                        r^{-\alpha_0 -1} S_0\bigl(\mathrm d(\mu_1,\mu_2)
                               \bigr)\\
=& \frac 1t \Bigl(\frac{x}{b_0(t)}\Bigr)^{-{\alpha_0}}
   \int_{\{(\mu_1,\mu_2):\mu_2-2a_u\mu_1>0\}} (\mu_2-2a_u\mu_1)^{\alpha_0}
   S_0\bigl(\mathrm d(\mu_1,\mu_2)\bigr).
\end{align*}
Of course, in practice $S_0, \alpha_0, b_0$ must be replaced by
estimators and $t$ is replaced by $n/k$ where $n$ is the sample size
of observations and $k$ is the number of observations used in estimation.}
{An example where we carry out these calculations is given for
  simulated data in Section \ref{subsubsec:rare}.}

\subsection{{Exploring for HRV with the Hillish estimator}.}\label{subsec:Hillish} The Hillish estimator was designed
for detection of the CEV model \cite{
heffernan:resnick:2005,heffernan:tawn:2004,heffernan:resnick:2007,  
lindskog:resnick:roy:2014,
das:mitra:resnick:2013,das:resnick:2011b,das:resnick:2011}
and extended to detecting hidden
regular variation in \cite{das:resnick:2015}. 
The generalized polar coordinate transform converts Cartesian
coordinates in the definition of regular variation into coordinates
satifying the CEV model. 
In this paper we show that the Hillish technique can detect HRV
when the  cone  removed from $\R^{2}_{+}$ or $\R^{2}$ 
 is $[\smallwedge]$. {The Hillish procedure is described below. 
First we define a conditional extreme value model.}

\subsubsection{The CEV model.}\label{subsub:cev}
Suppose the {random variables}  
 $(\xi,\eta) $ {form a random element of} $\R_+\times \R$ and there exists a regularly
varying function $b(t)\to\infty$ and a non-null measure {$\mu \in
\M((0,\infty)\times \R)$} and 
\begin{equation}\label{eqn:CEV}
 t \Prob\Bigl[\Bigl(\frac{\xi}{{b}(t)},{\eta}\Bigr) \in \;
   \cdot\;\Bigr] \to \mu(\cdot),\qquad \text{ in }\M ((0,\infty)\times \R).
\end{equation}
{Note that \eqref{eq:limMeasPolar} is of this form where only the first
component $\xi=d(\bZ,\bC_0)$ is scaled.
See also \eqref{eq:>smallwedge}  and
\eqref{eq:<smallwedge}. }
Additionally assume that 
\begin{itemize}
 \item[(a)] $\mu((r,\infty] \times  \cdot \,)$ is a non-degenerate
   measure 
for any fixed $r>0$, and,
 \item[(b)] $H(\cdot ):=\mu((1,\infty)\times \cdot \,)$ is a probability distribution. 
\end{itemize}
Then
 $(\xi,\eta)$ satisfies a \emph{conditional extreme value model} and we write
 $(\xi,\eta) \in \CEV({b},\mu)$. Note \eqref{eq:>smallwedge} and
 \eqref{eq:<smallwedge} are of the form given in \eqref{eqn:CEV}.
 {Hence the
   HRV statements are equivalent to the appropriate transforms of the
   variables following a CEV model.}  
 
\subsubsection{The Hillish procedure.}\label{subsub:hillish}
Now suppose 
 $(\xi_i,\eta_i); 1\le i \le n$ are iid replicates of $(\xi,\eta)$.
Define
$$
\begin{array}{llll}
\xi_{(1)} \ge \ldots \ge \xi_{(n)} & \text{The decreasing order
statistics of  $\xi_1,\ldots,\xi_n$.}\\[1mm]
\eta_i^*, ~ 1 \le i \le n & \text{The $\eta$-variable corresponding to
$\xi_{(i)}$, also called}\\[1mm] 
& \text{ the concomitant of $\xi_{(i)}$.}\\ [1mm]
N_{i}^k= \sum\limits_{l=i}^k \bone_{\{\eta_l^* \le \eta_i^*\}}& \text{Rank of $\eta^*_i$ among
$\eta_1^*,\ldots,\eta_k^{*}$. We write $N_i=N_i^k$.}\\ [1mm]
\eta_{1:k}^* \le \eta_{2:k}^* \le \ldots \leq \eta_{k:k}^* & \text{The
increasing order statistics of $\eta_1^*,\ldots,\eta_k^*$.}\\[1mm]
\end{array}
$$
{By analogy with the Hill estimator and the Hill plot, }
the {\it Hillish statistic\/} is defined for  $1\le k\le n$ as 
 \begin{align}\label{def:Hillish}
 \Hillish_{k,n}=\Hillish_{k,n}(\bbxi,\bbeta) := \frac{1}{k} \sum\limits_{j=1}^{k} \log \frac{k}{j} \log \frac{k}{N_{j}^k}
  \end{align}
According to \cite[Propositions 2.2 and 2.3]{das:resnick:2011b}, if
$(\xi,\eta)\in \CEV(b,\mu) $ then {there exists a limit $I_\mu$
  and }
$\Hillish_{k,n}\stackrel{P}{\to}
I_\mu$ and moreover, $\mu $ is a product measure iff both 
\begin{align}\label{eq:onetwo}
\Hillish_{k,n}(\bbxi,\bbeta) \cinP 1 \quad\text{and} \quad \Hillish_{k,n}(\bbxi,-\bbeta) \cinP 1, \end{align}
{as $k\to \infty, n\to \infty, n/k\to \infty$}. Note the limits in \eqref{eq:>smallwedge},
 \eqref{eq:<smallwedge} are all product measures. {Hence the diagnostic
   for detecting regular variation with a specified    forbidden zone
for the random vector $\bZ$
 is to
   plot  the Hillish   statistic of $\gpolar (\bZ)$.} 

{We emphasize
 that if \eqref{eq:onetwo} holds, we have empirical behavior {\it consistent\/}
 with the presence of regular variation but this does not prove
 existence of regular
 variation. When the Hillish technique fails because the plots do not
 hug the line at height 1, we can reject a hypothesis of
 regular variation. }


\section{Data Analysis with Simulated Data}\label{sec:dataAnal}
Before moving to real data, we test our analysis techniques} on two simulated data sets to
see how well they perform in Section \ref{ex:sim1} and \ref{ex:sim2}.
 Further, we
discuss  MRV and HRV properties and their detection. {In Section \ref{subsec:data},} we analyze two real bivariate data sets
both of which exhibit heavy-tailed margins and strong asymptotic dependence.

\subsection{Example 1: Full asymptotic dependence.}\label{ex:sim1} 
Suppose $Z_{1}\sim$ Pareto(1.5) and $Z_{2}\sim$ Pareto(2.5) and independent of each other. Let $B_{1}, B_{2}$ be iid Bernoulli (0.5) random variables also independent of $Z_{1}$ and $Z_{2}$. Now define the vector $\bX=(X_{1},X_{2})$ as 
\begin{align*}
X_{1}&=B_{1}Z_{1}+(1-B_{1})Z_{2},\\
X_{2}&=B_{1}Z_{1}+B_{2}(1-B_{1})(1.5Z_{2})+(1-B_{2})(1-B_{1})(0.5Z_{2})
\end{align*}

\begin{figure}[h] 
\centering
\includegraphics[width=6in]{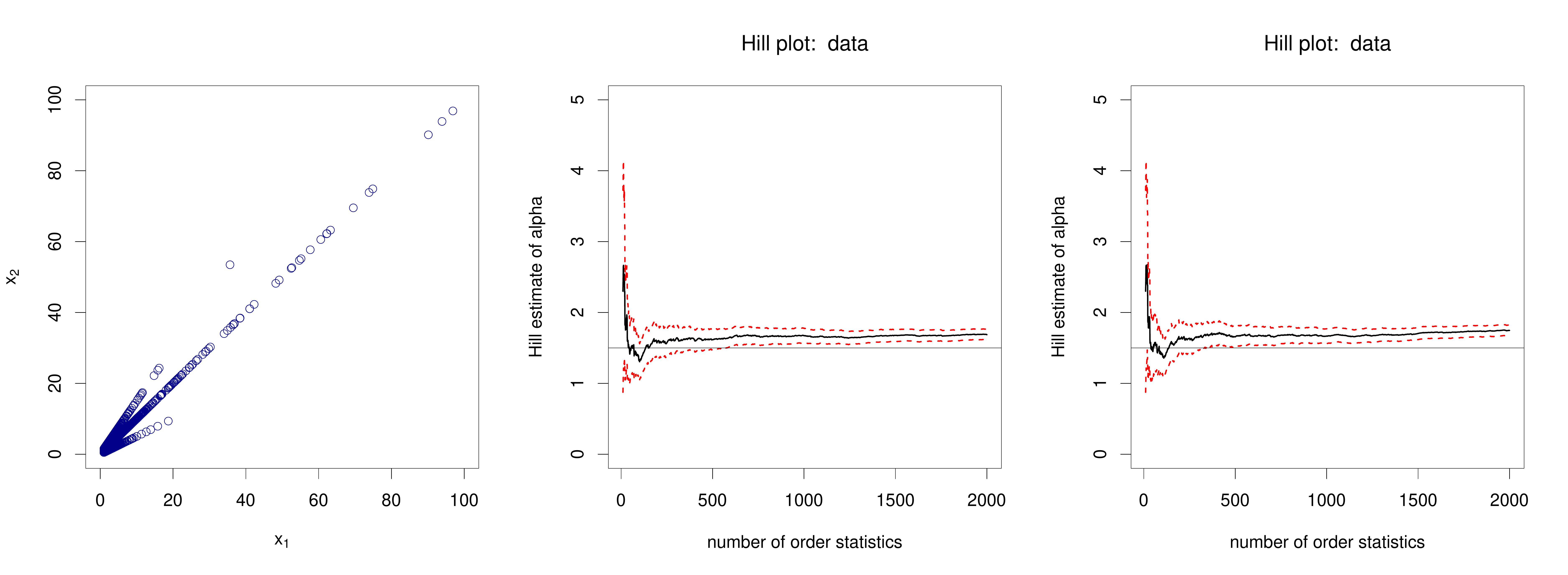}\qquad
\caption{Example \ref{ex:sim1}: (Left) Scatter plot of 10,000 data points. (Center
  and right) Hill plots for tail parameters of the marginal
  distribution{s of $X_1$ (center) and $X_2$ (right).}{The horizonal lines are at height $1.5.$}} 
\label{fig:Margins_sim}
\end{figure}
By construction 
$${
\bX= 
\begin{cases}
(Z_1,Z_1),& \text{ with probability }P[B_1=1]=\frac 12,\\
(Z_2,1.5Z_2) ,& \text{ with probability }P[B_1=0,B_2=1]=\frac 14,\\
(Z_2,0.5 Z_2), & \text{ with probability }P[B_1=0,B_2=0]=\frac 14,\\
\end{cases}
}$$
so that $\bX=(X_{1},X_{2})$ lies on $y=x$ with probability 0.5.
With probability 0.25 each it is either
on the line $y=0.5x$ or on $y={1.5}x$. In
this model, $\bX$ is MRV with parameter $\alpha=1.5$ and it has full
dependence on the diagonal $[\diag]:= \{(x,y)\in\R_{+}^{2}: y=x\}$.
On the other hand, on $\R^{2}_{+} \setminus [\diag]$, 
\begin{align*}
t\,\PR \left(\frac{d((X_{1},X_{2}), [\diag])}{t^{1/2.5}/2\sqrt{2}} > x, \frac{(X_{1},X_{2})}{\sqrt{2}d((X_{1},X_{2}), [\diag])} = \ba \right) & = t\,\PR \left(\frac{(2|X_{1}-X_{2}|}{t^{1/2.5}} > x, \frac{(X_{1},X_{2})}{|X_{1}-X_{2}|} = \ba \right)\\
& \to x^{{-2.5}} \times \frac 14 = \frac 12
  \nu_{{2.5}}(x,\infty)S_{{0}}(\{{\sqrt 2 \ba }\}).
\end{align*} 
where $\ba=(2,1)$ or $(2,3)$ {
and $S_0(\sqrt 2 \ba)=0.5$}. Hence we have HRV on $\R_+^{2}\setminus[\diag]$ with tail parameter $\alpha_{0}=2.5$.

\begin{figure}[h]
\includegraphics[width=6in]{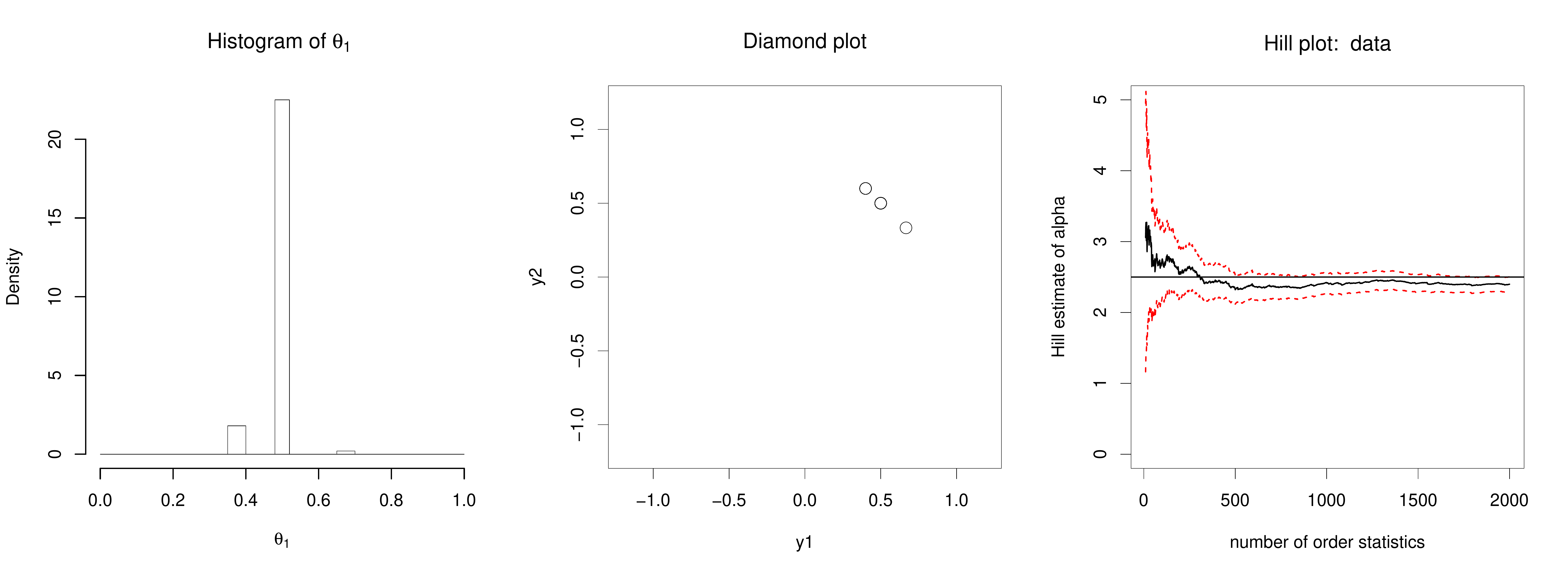}
\vspace{-.2in}
\caption{Example \ref{ex:sim1}: {Histogram} of $\theta_1$ (left), diamond plot (center),  and
  Hill plot for tail estimate of $|X_{1}-X_{2}|$ (right). {The
  horizontal line is at height 2.5}.}
\label{fig:diamond_sim1}
\end{figure}

We generate $n=10,000$ iid samples from this data set. The scatter plot in Figure \ref{fig:Margins_sim} shows the dependence structure of $\bX$ along with Hill plots of $X_{1},X_2$ which supports the premise that $\alpha=1.5$.

\begin{figure}[h]
\includegraphics[width=4.5in]{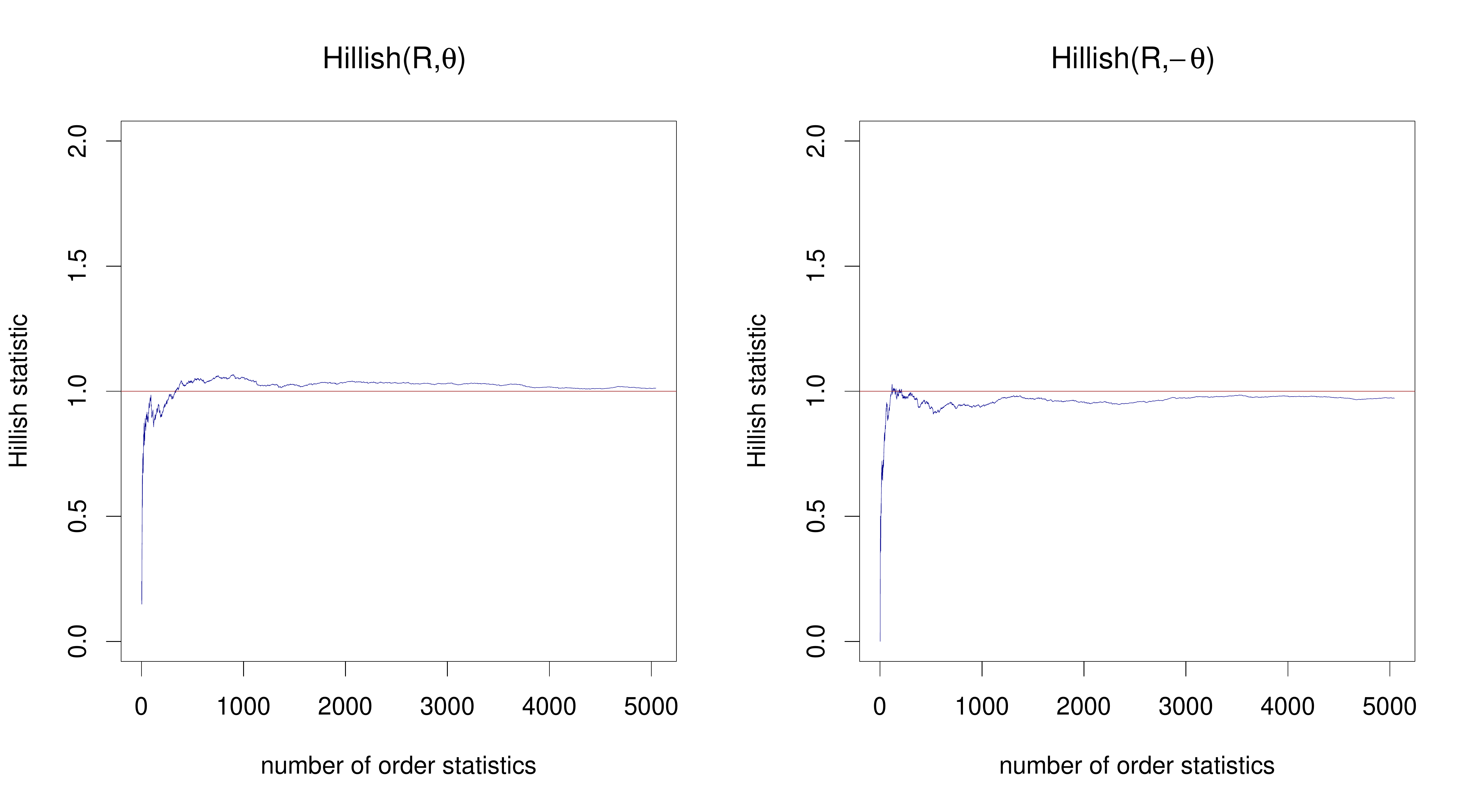}
\vspace{-.2in}
\caption{Example \ref{ex:sim1}: Hillish plots for $(\xi,\eta)$ (left) and
  $(\xi,-\eta)$ (right) where $(\xi,\eta)=(|X_1-X_2|,X_1/|X_1-X_2|)$.}
\label{fig:hillish_sim1}
\end{figure}

To understand the dependence structure of the variables $\bX$, we
{make} {the diamond plot, the} transformation from $\R^2\mapsto 
\aleph_{\bzero} \subset \R^2$ onto the $L_1$ unit sphere represented
by the diamond $\{(\theta_1,\theta_2):|\theta_1|+|\theta_u|=1\}$
We do the mapping at various thresholds determined by $k$, the number
of order statistics of the norms 
$\|\bx\|=|x_1|+|x_2|$. In Figure \ref{fig:diamond_sim1}, the diamond plot and  histogram of the angular measure is shown for $k=100$. Clearly the data is concentrated at $x=y$. The Hill plot of the quantity $|X_1-X_2|$ supports the fact that data was generated with hidden tail parameter $\alpha_{0}=2.5$.

Finally, we look at the Hillish statistic  for
$(\xi,\eta)=(|X_1-X_2|,X_1/|X_1-X_2|)$ after removing $[\diag]$.
 The Hillish plot in Figure
\ref{fig:hillish_sim1} is convincingly stable and close to 1,
supporting the presence of HRV as expected from the generation
procedure here.

\subsection{Example 2: Strong asymptotic dependence.}\label{ex:sim2} 
{In this example, we simulate data from a model with strong asymptotic dependence. We also exhibit estimation of rare probabilities and conduct a sensitivity analysis when the support of the regular variation for the the first level, and hence that for HRV is not correctly identified.} 

\begin{figure}[h] 
\centering
\includegraphics[width=6.5in]{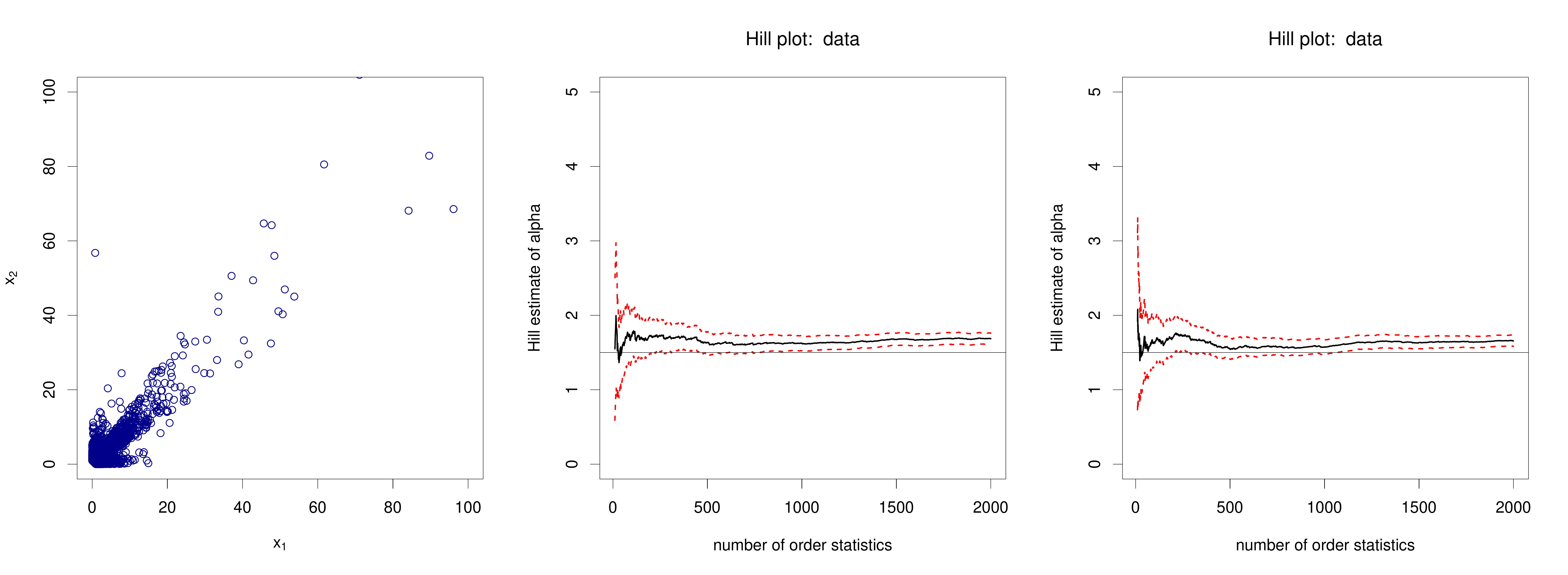}\qquad
\caption{Example \ref{ex:sim2}: (Left) Scatter plot of 30,000 data points. (Center
  and right) Hill plots for tail parameters of the margin{al
    distributions } of
  $X_1,X_2$. {The horizontal lines are at height $1.5$.}} 
\label{fig:Margins_sim2}
\end{figure}

Suppose $R_{1}\sim$ Pareto(1.5) and $R_{2}\sim$ Pareto(2.5) and
independent of each other. Let $\Theta_{1}\sim \Unif[0.4,0.6]$,
$\Theta_{2}\sim \Unif([0,1]\setminus[0.4,0.6)]$, $B\sim$ Bernoulli
(0.5) random variables. Assume the random variables are 
all independent.
  Now define the vector $\bX=(X_{1},X_{2})$ as  
\begin{align*}
X_{1}&=BR_{1}\Theta_{1} + (1-B)R_{2}\Theta_{2},\\
X_{2}&=BR_{1}(1-\Theta_{1}) + (1-B)R_{2}(1-\Theta_{2}).
\end{align*}

By construction, $\bX$ is MRV on $\R_{+}^{2}\setminus\{\bzero\}$ with tail parameter $\alpha=1.5$. 
{Corresponding to $(\theta_l,\theta_u)=(0.4,0.6)$, we have by
\eqref{e:parCone} that $(a_l,a_u)=(0.67,1.5)$ and therefore}
\[[\smallwedge] := \{(x,y)\in \R_{+}^{2}: 0.67 x \le y \le 1.5 x\}.\]
{This gives hidden regular variation with tail parameter
$\alpha_{0}=2.5$
 on $\R_{+}^{2}\setminus[\smallwedge]$.}

\begin{figure}[h]
\includegraphics[width=6.5in]{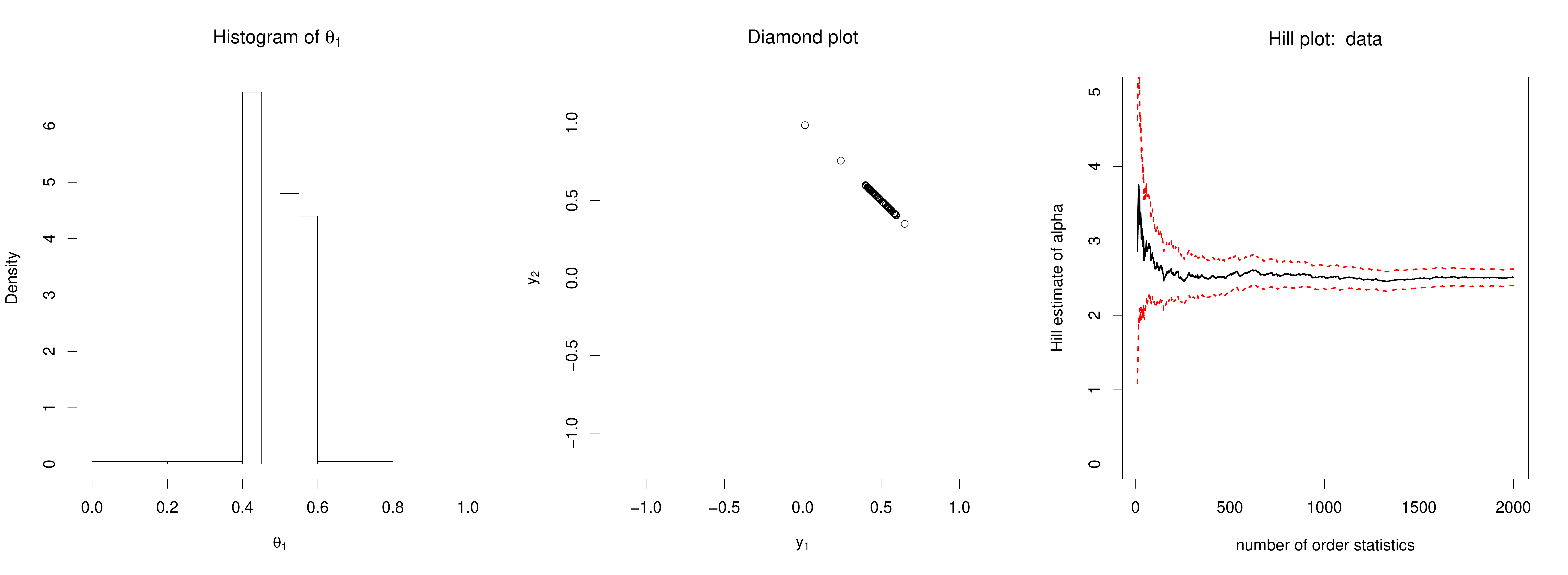}
\vspace{-.2in}
\caption{Example \ref{ex:sim2}: Diamond plot, histogram and Hill plot for tail estimate of
  $d(\bX,[\smallwedge])$. {The horizontal line in the Hill plot is at
  height 2.5.}}
\label{fig:diamond_sim2}
\end{figure}

\begin{figure}[h]
\includegraphics[width=6.5in]{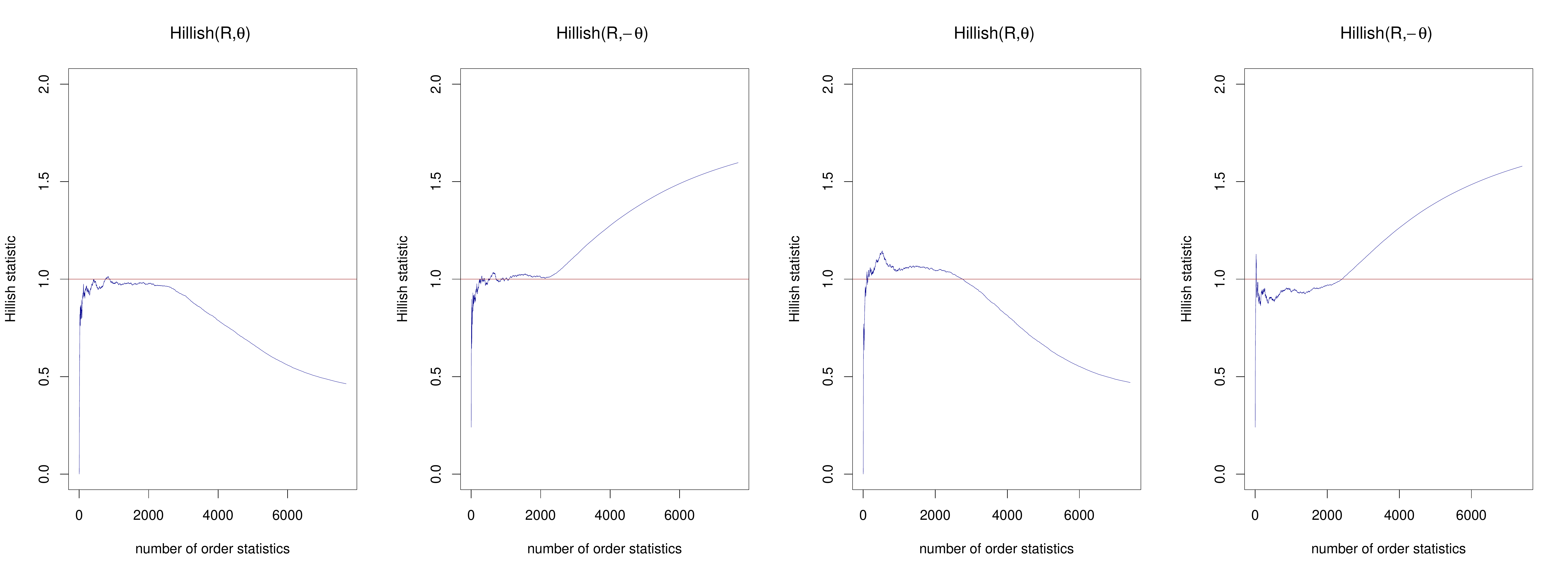}
\vspace{-.2in}
\caption{{Example \ref{ex:sim2}: Hillish plots for (i)$(\xi_{1},\eta_{1}),$ (ii)$(\xi_{1},-\eta_{1})$,  (iii)$(\xi_{2},\eta_{2}),$ (iv)$(\xi_{2},-\eta_{2})$  respectively where $(\xi_{1},\eta_{1})$ and $(\xi_{2},\eta_{2})$ are obtained by using   \eqref{eq:>smallwedge} and \eqref{eq:<smallwedge} on $\bX=(X_{1},X_{2})$.} }
\label{fig:hillish_sim2}
\end{figure}
 
  \begin{figure}[h]

\includegraphics[width=6.5in]{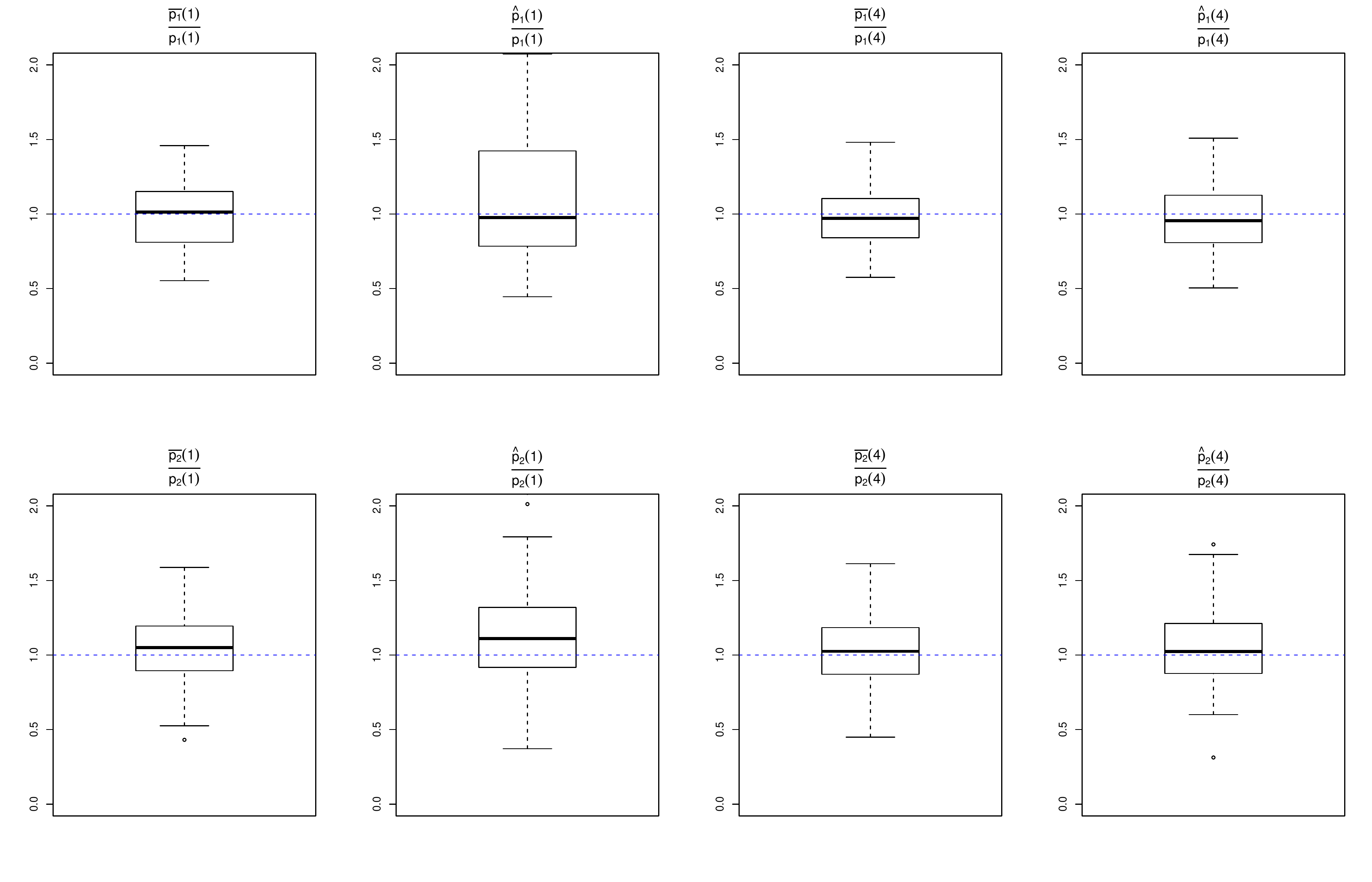}
\vspace{-.2in}
\caption{Example \ref{ex:sim2}: Boxplots for $\frac{\bar{p}_{1}(1)}{p_{1}(1)},\frac{\hat{p}_{1}(1)}{p_{1}(1)},\frac{\bar{p}_{1}(4)}{p_{1}(1)},\frac{\hat{p}_{1}(4)}{p_{1}(4)}$ (top) and $\frac{\bar{p}_{2}(1)}{p_{2}(1)},\frac{\hat{p}_{2}(1)}{p_{2}(1)},\frac{\bar{p}_{2}(4)}{p_{2}(4)},\frac{\hat{p}_{2}(4)}{p_{2}(4)}$ (bottom) .}
\label{fig:estimates_sim2}
\end{figure}

We generate $n=30,000$ iid samples from this data set. The scatter plot in Figure \ref{fig:Margins_sim2} shows the dependence structure of $\bX$ along with Hill plots of $X_{1},X_2$ which supports the premise that $\alpha=1.5$.

To understand the dependence structure of  $\bX$, we {graph} the
diamond plot as used in the previous example. We do the mapping at
various thresholds determined by $k$, the number 
of order statistics of the norms 
$|x|+|y|$.
In Figure \ref{fig:diamond_sim2}, the  histogram of angles and the diamond
plot are shown for $k=100$ and shows the angles are Uniform in
[0.4, 0.6] for high values of $|x|+|y|$. The Hill plot of the quantity
$d(\bX,[\smallwedge])$ supports the fact that data was generated with
hidden tail parameter $\alpha_{0}=2.5$. 
Finally, we look at a Hillish statistic  for 
{$(\xi_{1},\eta_{1})$ and $(\xi_{2},\eta_{2})$ respectively which are  obtained by using   \eqref{eq:>smallwedge} and \eqref{eq:<smallwedge} on $\bX=(X_{1},X_{2})$ after removing $[\smallwedge]$.} The Hillish
plots in Figure \ref{fig:hillish_sim2} are again convincingly stable and
close to 1 {and detect the hidden regular variation.}

\subsubsection{Probabilities of rare sets for this example.}\label{subsubsec:rare}
Now to further illustrate our methods, we  compute $\PR(X_{2}-2X_{1}>x)$ and $\PR(X_{2}-3X_{1}>x)$.
Without resorting to hidden regular variation we have $\bX$ is MRV on $\R_{+}^{2}\setminus\{\bzero\}$ with tail parameter $\alpha=1.5$
and the limit measure concentrates on
\[[\smallwedge] := \{\bx\in \R_{+}^{2}: a_{l} x_1 \le x_2 \le a_{u}
x_1\} = \{\bx \in \R_{+}^{2}: 0.67 x_1 \le x_2 \le 1.5 x_1\}.\]

Hence with {the usual regular variation} techniques we would estimate both
$$ \PR(X_{2}-2X_{1}>x) \approx 0 \quad \text{and} \quad \PR(X_{2}-3X_{1}>x) \approx 0.$$

But for this example we can compute the exact answer without resorting
to asymptotic approximations and we get,
\begin{align}
p_{1}(x):=& \PR(X_{2}-2X_{1}>x)  = \frac 12 \PR(R_{1}(1-\Theta_{1})
                                 {-} 2R_{1}\Theta_{1} > x) +
                                 \frac 12 \PR(R_{2}(1-\Theta_{2}) {-}
                                 2R_{2}\Theta_{2} > x)\nonumber \\ 
\intertext{and because $3\Theta_1>1$, this is}
                                   =&  \frac 12
                                     \PR(R_{2}(1-3\Theta_{2}) > x) =
                                     \frac 5{84} x^{{-2.5}}. \label{e:p1x}
\intertext{
Similarly we can compute} 
p_{2}(x):=&\PR(X_{2}-{3}X_{1}>x)  = \frac 5{112}
            x^{{-2.5}}.\label{e:p2x} 
\end{align}

     \begin{figure}[h]
\includegraphics[width=6in]{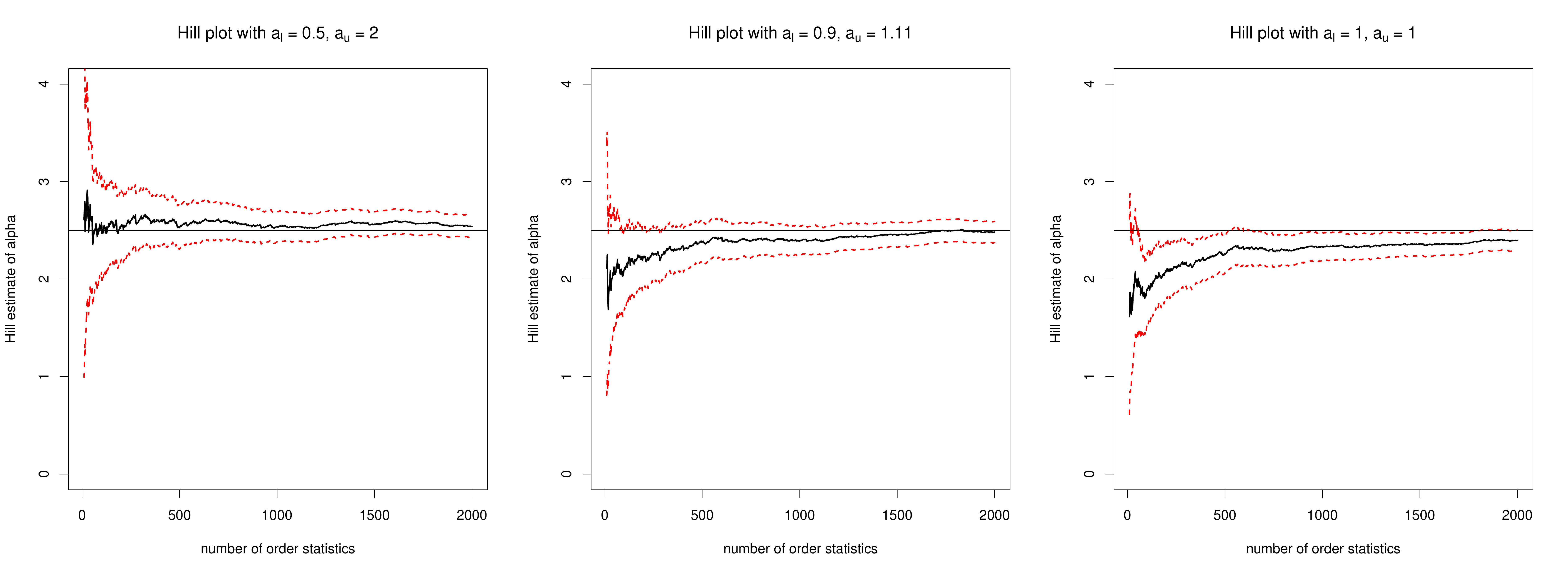}
\vspace{-.2in}
\caption{{Example \ref{ex:sim2}: Hill plots for estimating
    $\alpha_{0}$ (known to be 2.5 in the model) when the support set
    of MRV is incorrectly specified as  $(a_{l}, a_{u})= (0.5,2)$
    (left), $(a_{l}, a_{u})= (0.9,1.11)$ (middle), $(a_{l}, a_{u})=
    (1,1)$ (right) respectively. The horizontal line in the Hill plot
    is at height 2.5.}} 
\label{fig:sim2_alpha0_err}
\end{figure}

     \begin{figure}[h]
\includegraphics[width=6in]{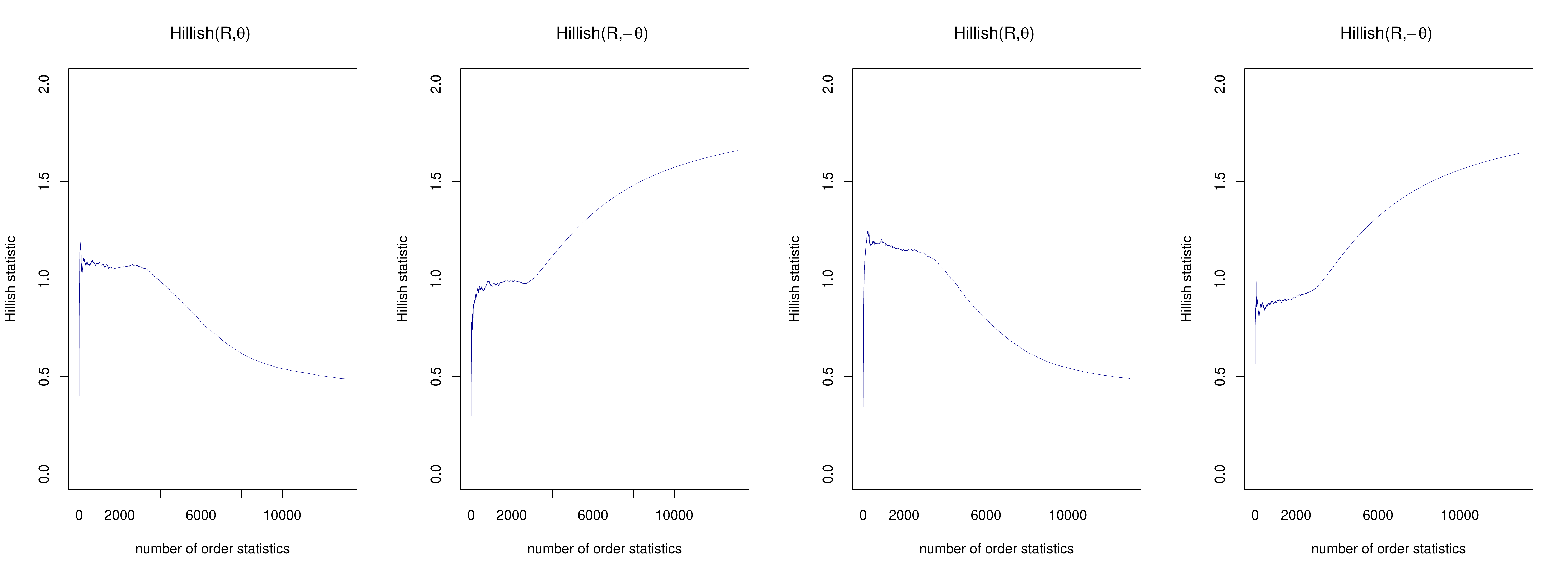}
\vspace{-.2in}
\caption{{Example \ref{ex:sim2}:  Hillish plots for (i)$(\xi_{1},\eta_{1}),$ (ii)$(\xi_{1},-\eta_{1})$,  (iii)$(\xi_{2},\eta_{2}),$ (iv)$(\xi_{2},-\eta_{2})$  respectively where $(\xi_{1},\eta_{1})$ and $(\xi_{2},\eta_{2})$ are obtained by using   \eqref{eq:>smallwedge} and \eqref{eq:<smallwedge} on $\bX=(X_{1},X_{2})$ and the support of MRV is (incorrectly) identified by $(a_{l},a_{u})=(0.9,1.11)$.}}
\label{fig:sim2_hillish_err}
\end{figure}
 
  \begin{figure}[h]
\includegraphics[width=6in]{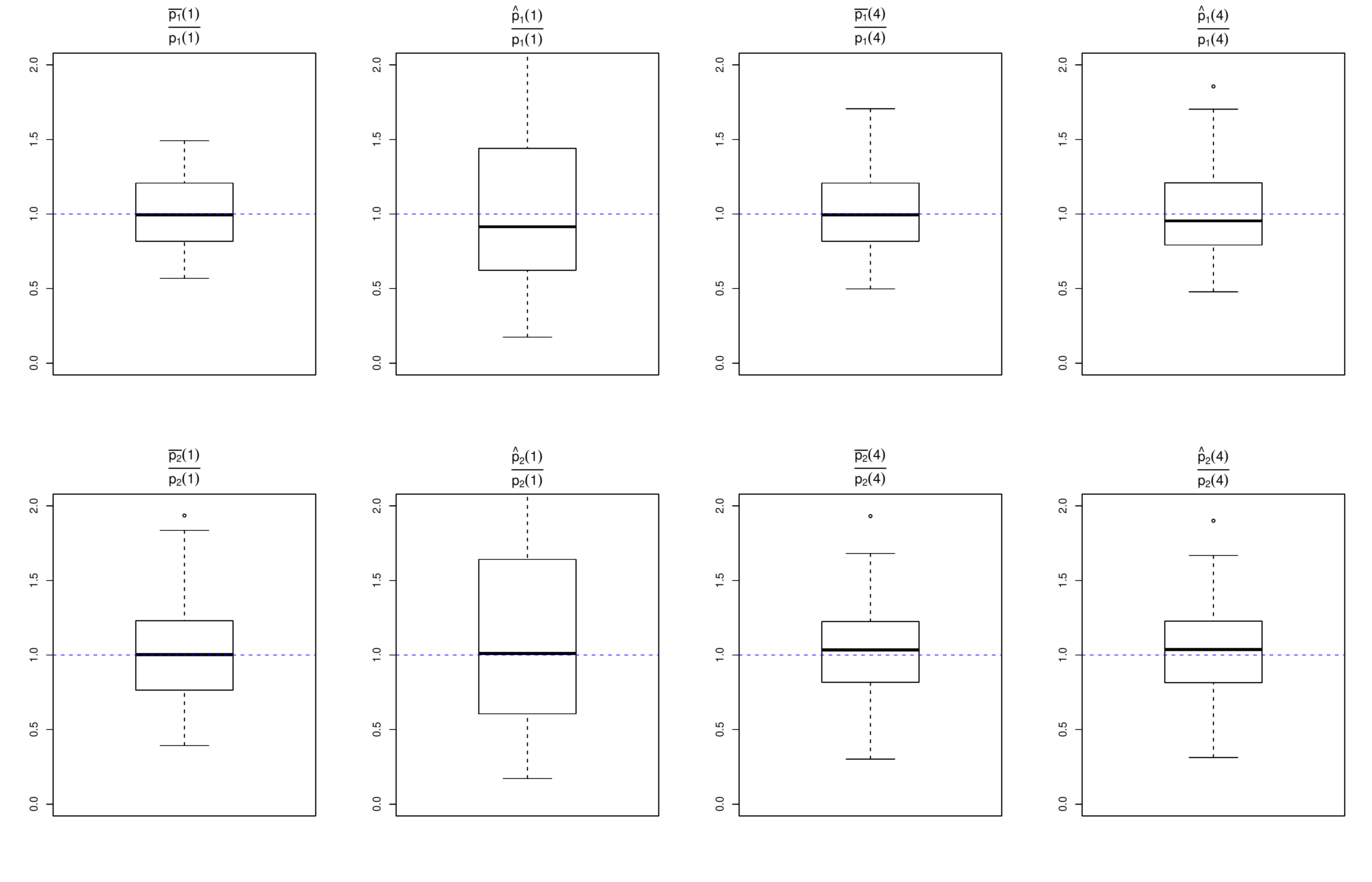}
\vspace{-.2in}
\caption{Example \ref{ex:sim2}: Boxplots for $\frac{\bar{p}_{1}(1)}{p_{1}(1)},\frac{\hat{p}_{1}(1)}{p_{1}(1)},\frac{\bar{p}_{1}(4)}{p_{1}(4)},\frac{\hat{p}_{1}(4)}{p_{1}(4)}$ (top) and $\frac{\bar{p}_{2}(1)}{p_{2}(1)},\frac{\hat{p}_{2}(1)}{p_{2}(1)},\frac{\bar{p}_{2}(4)}{p_{2}(4)},\frac{\hat{p}_{2}(4)}{p_{2}(4)}$ (bottom) with (incorrect) identification of support where $a_l=0.5, a_{u}=2$.}
\label{fig:estimates_sim2_alt1}
\end{figure}

  \begin{figure}[h]
\includegraphics[width=6in]{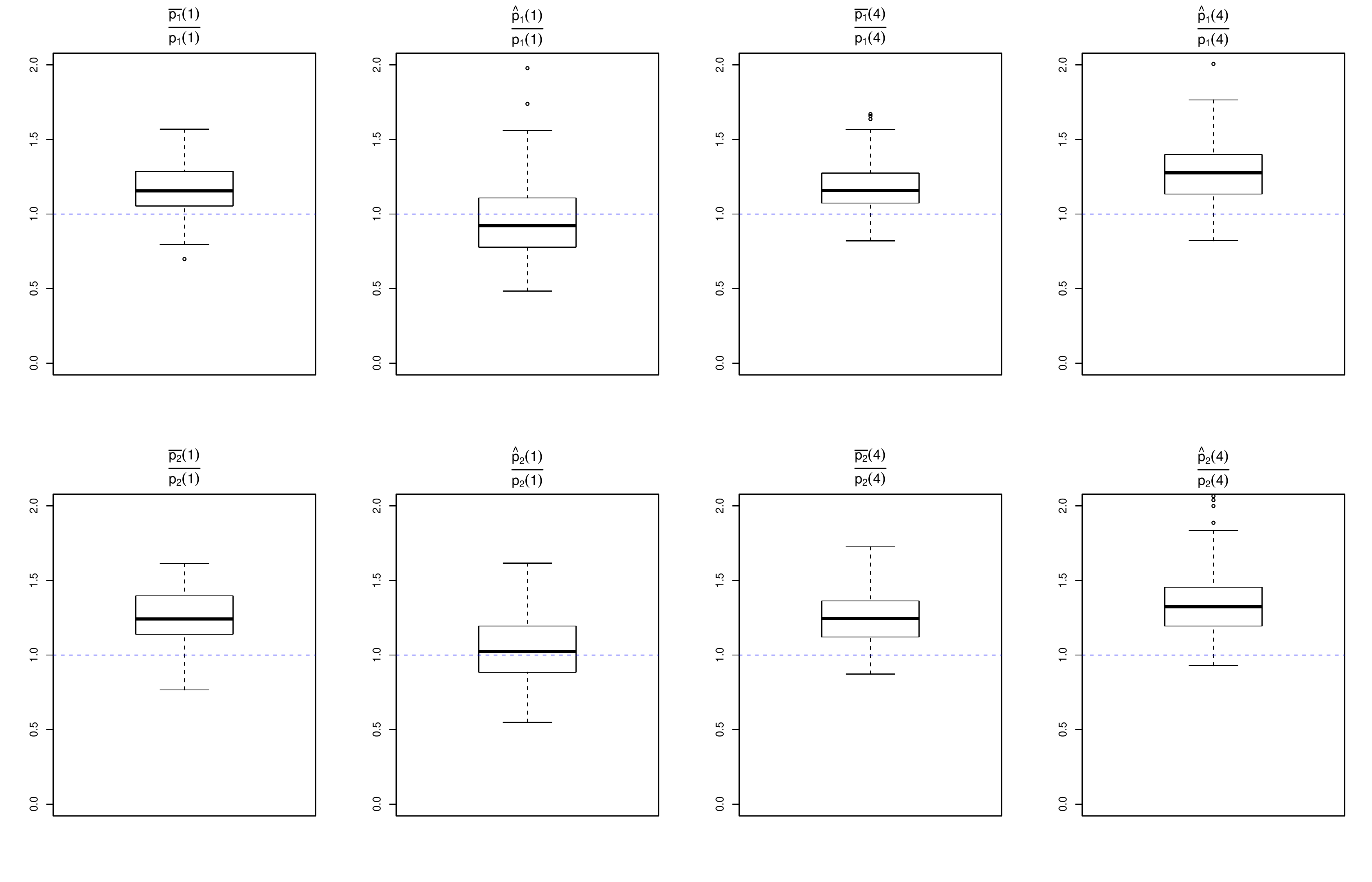}
\vspace{-.2in}
\caption{Example \ref{ex:sim2}: Boxplots for $\frac{\bar{p}_{1}(1)}{p_{1}(1)},\frac{\hat{p}_{1}(1)}{p_{1}(1)},\frac{\bar{p}_{1}(4)}{p_{1}(1)},\frac{\hat{p}_{1}(4)}{p_{1}(4)}$ (top) and $\frac{\bar{p}_{2}(1)}{p_{2}(1)},\frac{\hat{p}_{2}(1)}{p_{2}(1)},\frac{\bar{p}_{2}(4)}{p_{2}(4)},\frac{\hat{p}_{2}(4)}{p_{2}(4)}$ (bottom) 
 with (incorrect) identification of support where $a_l=0.9, a_{u}=1.1$.}
\label{fig:estimates_sim2_alt2}
\end{figure}

{If we pretend we do not know the exact answers provided by 
\eqref{e:p1x}, \eqref{e:p2x}, can we give}  better estimates than 0
using
asymptotic methods on our simulated data
set?
 Under hidden regular variation after removing
 $\bC_{0}=[\smallwedge]$, we know that \eqref{e:delC0} holds for $\bX$
 and hence 
\begin{equation}\label{e:delC000}
\frac nk\Prob \Bigl[ \frac{d({\bX}, \bC_{0})}{b_0(\frac nk)} > x,
\frac{\bX}{d(\bX,\bC_0)}\Bigr) \in \Lambda \,\Bigr]
\to x^{{-\alpha_{0}}}  \times
S_0(\Lambda) 
\end{equation}
as $n\to\infty$, $k\to \infty$, $k/n\to0$, where $S_0(\cdot) $ is a probability measure and  $\Lambda\subset\aleph_{\mathbb{C}_0}$.  From Section \ref{subsec:riskestimate}, 
we have for $x>0$, as $n\to \infty$, $k\to\infty$, $k/n \to 0$,
\begin{align}\label{estimate}
\frac nk p_1(b_0(\frac nk) x)=
\frac nk \PR[X_{2}-2X_{1}>b_0(\frac nk )x] \to  x^{-{\alpha_0}}
   \int_{\{(\mu_1,\mu_2):\mu_2-2\mu_1>0\}} (\mu_2-2\mu_1)^{\alpha_0}
   S_0\bigl(d(\mu_1,\mu_2)\bigr),
   \end{align}
with a similar limiting expression for $p_2(x)$. This suggests we need to estimate $\alpha_0$, $b_0(n/k)$,
$S_0(\cdot)$ and of course the wedge $\mathbb{C}_0$.

For the wedge, we
estimate $\hat a_{l}, \hat a_{u}$ using the $5th$ and $95th$
percentile of the range of $X_{1}/(X_{1}+X_{2})$ for 100 highest
values of $X_{1}+X_{2}$ for each simulation.
In \eqref{e:delC000} replace $x$ by 1 and $\Lambda$ by
$\aleph_{\mathbb{C}_0}$ and then we estimate $b_0(n/k)$ with the $k$th
largest value of $d(\bX_i,\mathbb{C}_0) $ corresponding to $\bX's $
outside $\mathbb{C}_0. $  Alternatively, if we fix $b_0(n/k)$ then, we obtain the appropriate $k$
largest value of $d(\bX_i,\mathbb{C}_0) $ corresponding to $\bX's $
outside $\mathbb{C}_0.$ to be used. For $S_0(\cdot)$ we modify the argument
leading to \cite[Eq. 9.47, p. 313]{resnickbook:2007}. 

Corresponding to  estimates $\widehat{\mathbb{C}}_0$, $\hat
       b_0(n/k)$, and choice of $k$, we enumerate the
       gpolar-transformed points outside $\widehat{\mathbb{C}}_0$,
       corresponding to the $k$ largest values of
       $d(\bX_i,\widehat{\mathbb{C}}_0)$ as
       $\{(r_{i},\mu_{1i},\mu_{2i}): 1\le i \le k\}$. Then using these
       points we estimate $S_0 (\cdot)$ with the empirical
       distribution as
$$\hat S_0(\cdot)=\frac 1k \sum_{i=1}^k \epsilon_{(\mu_{1i},\mu_{2i})} (\cdot).$$
This leads to the risk estimates,
\begin{align}
 \hat{p}_{1}(x) =& {x^{-{\hat \alpha_0}}  } \frac kn (\hat b_0(n/k))^{\hat{\alpha_{0}}}
                   \frac 1k \sum_{i=1}^{k}
                   (\mu_{2i}-2\mu_{1i})^{\hat{\alpha_{0}}}\bone_{\{\mu_{2i}>2\mu_{1i}\}}
                   , \label{newestimate}\\
 \hat{p}_{2}(x)=&     {x^{-{\hat \alpha_0}}  } \frac kn (\hat b_{0}(n/k))^{\hat{\alpha_{0}}} \frac 1k \sum_{i=1}^{k} (\mu_{2i}-{3}\mu_{1i})^{\hat{\alpha_{0}}}\bone_{\{\mu_{2i}>2\mu_{1i}\}}.\label{newestimate2}
   \end{align}

Since $\alpha_{0}, a_{l}, a_{u}$ are known for this simulation
example, we may compare $\hat 
p_i(x)$ with $\bar p_i(x)$ estimated using the three known values.
We carry out comparisons using  $x=1$ and $x=4$. We conduct simulations
with $n=10,000$  and use a value of $k$ corresponding to $b_0(n/k)=2$. We compute
$\hat{p}_{1}(1),\hat{p}_{1}(4),\bar{p}_{1}(1),\bar{p}_{1}(4),$ for 100
iterations and create box plots of
$\frac{\hat{p}_{1}(1)}{p_{1}(1)},$
$\frac{\hat{p}_{1}(4)}{p_{1}({4})}$,
$\frac{\bar{p}_{1}({1})}{p_{1}(1)}$, $\frac{\bar{p}_{1}(4)}{p_{1}(4)}$
and we do the same for $p_{2}(1)$ and $p_{2}(4).$ From Figure  
\ref{fig:estimates_sim2}, clearly the estimates perform pretty well
since the ratios of the estimates to the real values are very close to
1. Clearly, when $\alpha_{0}$ is estimated the error bounds become
larger, but still perform reasonably. Note that the quantities we
compute have low probabilities: 
$$p_{1}(1)=0.{0}595,\quad p_{2}(1)=0.044, \quad p_{1}(4)=0.002,
\quad p_{2}(4)=0.0014.$$ 

To summarize: This estimation procedure can be used to calculate
risk probabilities in the presence of hidden regular variation when
the primary regular variation gives a 
zero risk estimate.
   
   \subsubsection {Sensitivity analysis in this example}
   {Clearly, the probability estimation procedure we discussed hinges on our ability to appropriately estimate the support set of 
 regular variation at the first level, given by 
  \[[\smallwedge] := \{\bx \in \R_{+}^{2}: a_{l} x_{1} \le x_{2} \le a_{u} x_{1}\}.\]

 An inaccurate estimation of the support leads to  an improper estimation of $\alpha_{0}$ and hence also the probabilities of rare events.  Under the current  Example \ref{ex:sim2},  we conduct a sensitivity analysis of our estimation procedures  by choosing a support set which is different from the one that is specified by the model.  In the example, the support set is identified by $(a_{l}^{*},a_{u}^{*})=(0.67,1.5)$. Recall that we have 30,000 data points from this model.

    First we estimate $\alpha_{0}$ under an improper specification of
    $a_{l}$ and $a_{u}$. This is estimated using a Hill plot of points
    comprised of $(X_{2}-a_{u}X_{1})/\sqrt{1+a_{u}^{2}}$ for
    $X_{2}-a_{u}X_{1}>0$ and  $(a_{l}X_{1}-X_{2})/\sqrt{1+a_{l}^{2}}$
    for $a_{l}X_{1}-a_{u}X_{2}>0$ for different choices of $(a_{l},
    a_{u})$. In Figure \ref{fig:sim2_alpha0_err}, we provide Hill
    plots for estimating $\alpha_{0}$ by using   $(a_{l}, a_{u})=
    (0.5,2), (a_{l}, a_{u})= (0.9,1.11), (a_{l}, a_{u})= (1,1)$
    respectively. Comparing the plots with the one in Figure
    \ref{fig:diamond_sim2}, the estimates clearly move away from the
    actual value $\alpha_{0}=2.5$ as the size of the support
    decreases. When we take $(a_{l},a_{u}) = (0.5,2)$ the data points
    used to estimate $\alpha_{0}$ are a subset of the points used to
    estimate $\alpha_{0}$ when $(a_{l}^{*},a_{u}^{*})=(0.67,1.5)$, and
    thus are regularly varying with parameter $\alpha_{0}=2.5$; hence
    the Hill plot clearly hugs the horizontal line at $y=2.5$. As the
    size of the support set decreases we see the Hill estimates become
    lower than $2.5$ and moves towards $\alpha=1.5$. We also observe
    in Figure \ref{fig:sim2_hillish_err}  that the Hillish plots are
    not  that close to the horizontal line at height 1 when the
    support is not correctly identified; in this case
    $(a_{l},a_{u})=(0.9,1.11)$. In comparison, Figure
    \ref{fig:hillish_sim2} clearly shows that  the Hillish plots are
    close to 1, when the support is correctly specified.  
  
  Finally we estimate probabilities $p_{1}(x)=\PR(X_{2}-2X_{1}>x)$ and
  $p_{2}(x)=\PR(X_{2}-3X_{1}>x)$ when the support sets are identified
  incorrectly. Figure \ref{fig:estimates_sim2_alt1} corresponds to
  boxplots of
  $\frac{\bar{p}_{1}(1)}{p_{1}(1)},\frac{\hat{p}_{1}(1)}{p_{1}(1)},\frac{\bar{p}_{4}(4)}{p_{1}(1)},\frac{\hat{p}_{1}(4)}{p_{1}(4)}$
  and
  $\frac{\bar{p}_{2}(1)}{p_{2}(1)},\frac{\hat{p}_{2}(1)}{p_{2}(1)},\frac{\bar{p}_{2}(4)}{p_{2}(4)},\frac{\hat{p}_{2}(4)}{p_{2}(4)}$
  where $\hat{p}_{i}, i=1,2$ uses the estimators given in
  \eqref{newestimate},\eqref{newestimate2} with $\alpha_{0}$ computed
  using $(a_{l},a_{u})=(0.5,2)$ and  $\bar{p}_{i}, i=1,2$ uses
  $\alpha_{0}=2.5, a_{l}=0.5,a_{u}=2$. Observe that the boxplots are
  quite close to 1, since  $\alpha_{0}$ is estimated well. On the
  other hand in Figure \ref{fig:estimates_sim2_alt2}, the similar
  boxplots are done for $(a_{l},a_{u})=(0.9,1.11)$. In this case,
  $\alpha_{0}$ is not that well-estimated and hence the boxplots are
  clearly away from 1. In both cases, 100 replications of data sets
  with 10,000 data points in each were used to create the boxplots. 
  
  In conclusion we can see that an incorrect identification of the support of regular variation can often lead to incorrect estimates. Although if the identified support of hidden regular variation is a bit smaller than the correct one (which means that the identified support of MRV is larger than the correct one), then the estimates are still quite accurate.}

\newpage
\section{Examples of strong asymptotic dependence with real data.}\label{subsec:data}
We now analyze two real data sets: (i) facebook wall posts and (ii)
returns from Exxon and Chevron.

\subsection{Facebook wall posts}\label{subsec:facebk}

\begin{figure}[h]
\centering
\includegraphics[width=.4\textwidth]{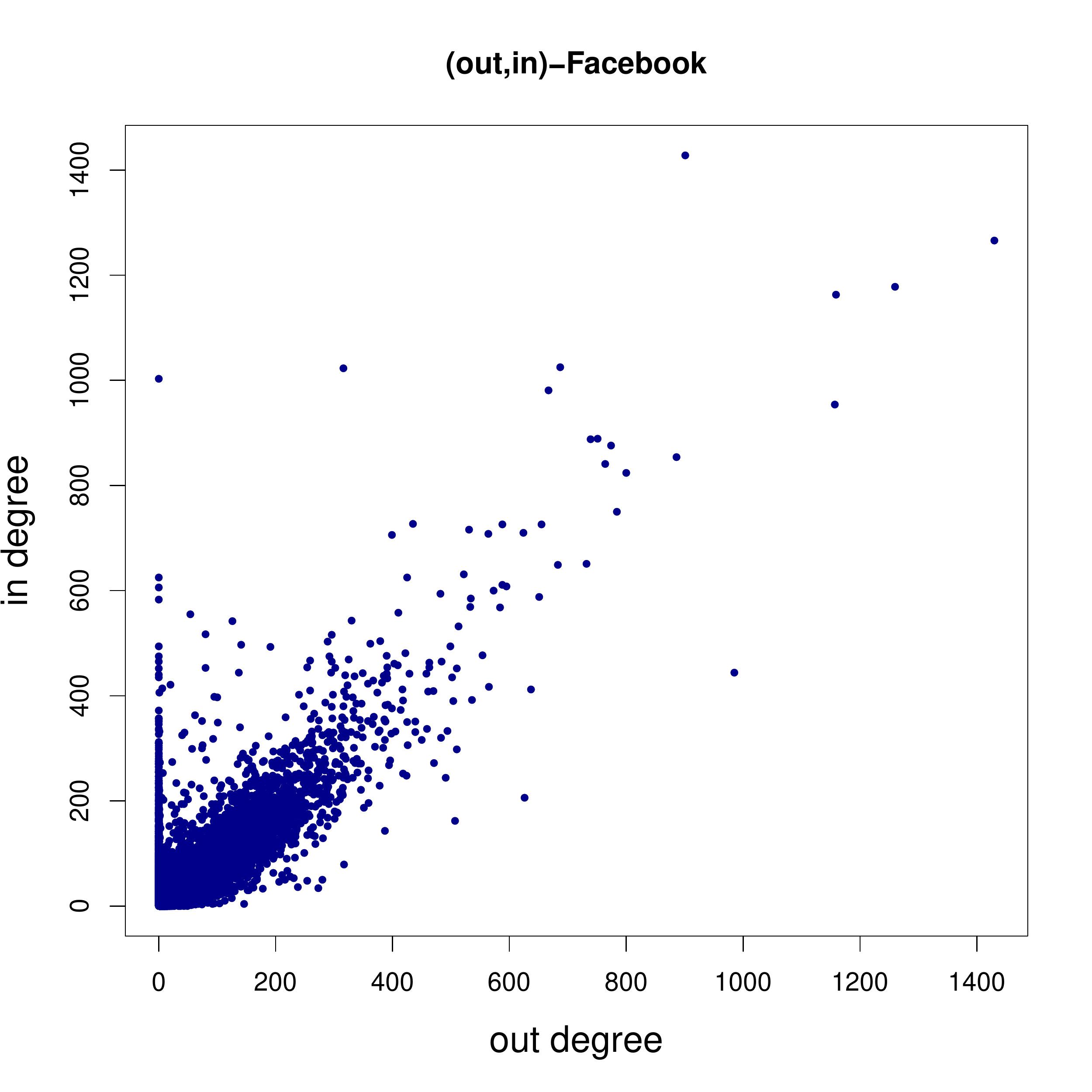}
\vspace{-.23in}
\caption{Scatter plot of node-wise out-degree and in-degree of Facebook wallpost graph.}
\label{fig:fbwallscatter}
\end{figure}

The Facebook wall posts data was downloaded  from
\url{http://konect.uni-koblenz.de/networks/facebook-wosn-wall} and
has been analyzed in \cite{viswanath:2009}. 
Conversion of edge data to
node-indexed in- and out-degree counts was done
using the 
R-package {\it igraph\/} \cite{csardi:nepusz:2006}. 
The data is a
directed network representing posts by Facebook users to other users'
walls.  Nodes are users and a directed edge represents one post from the
user to the user whose wall is receiving the post. There are 46,952
users and 876,993 edges. We focus on out- and in-degree {indexed by the nodes as $\{(Z_{1,i},Z_{2,i}): 1\le i \le 46952\}$}. Of course
this data is not the result of iid replication but is rather
node-indexed; however, for reasons still being investigated, 
conventional tools of heavy tail analysis seem quite effective on
node-indexed network data. The  scatter-plot of (out,in)-degrees in Figure \ref{fig:fbwallscatter}
shows the expected strong asymptotic dependence between out- and  in-degrees.

\begin{figure}[h] 
\centering
\includegraphics[width=6.5in]{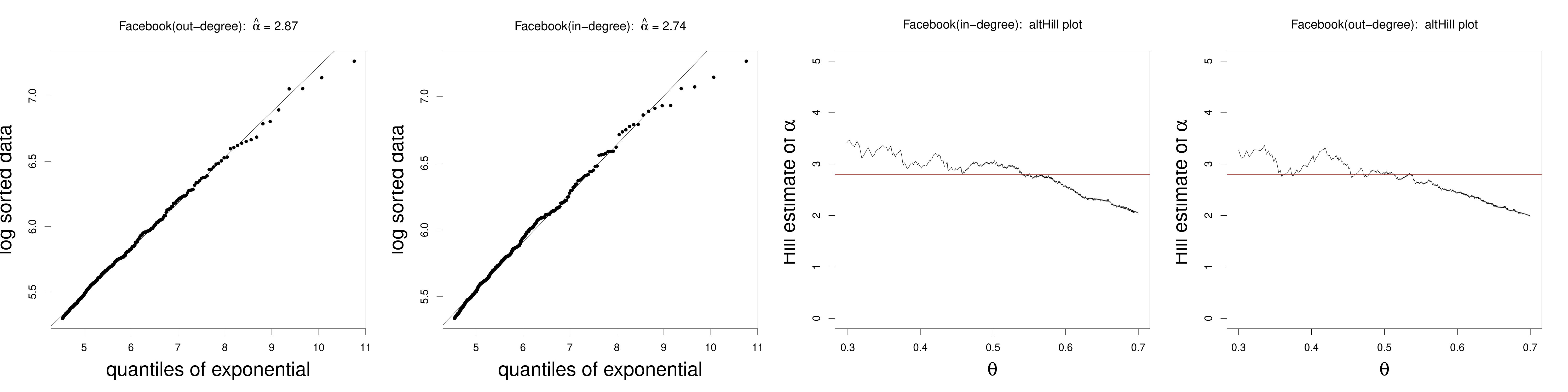}
\caption{(Left) Marginal estimation of tail indices by QQ plot slope
  estimation and (right) altHill plotting for out- and
  in-degree. Horizontal lines are at height 2.8.}
\label{fig:MarginalAnal}
\end{figure}

The plots in Figure \ref{fig:MarginalAnal} give the estimation of
distribution tail
indices for out- and in-degree. The slope estimator based on QQ-plots
(\cite{beirlant:vynckier:teugels:1996,kratz:resnick:1996, resnickbook:2007}) gives approximately
$\alpha=2.8$ for both out- and in-degree. Note this estimate is for the
tail of the cumulative distribution functions and not, as is customary
in network science, the index of the power law of the mass
functions. The Hill estimator is ineffective and we have provided
altHill plots (\cite{resnickbook:2007,resnick:starica:1997a,
  drees:dehaan:resnick:2000}).

\begin{figure}[h]
\centering
\includegraphics[width=5in]{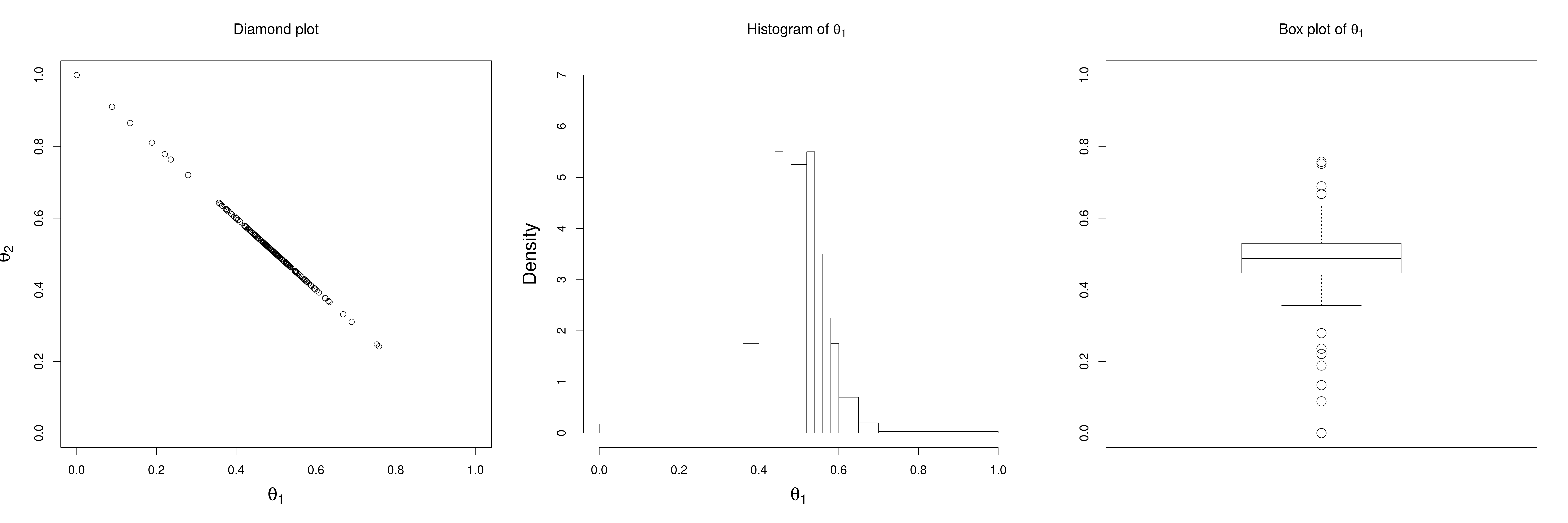}
\caption{Diamond plot restricted to the first quadrant for empirical angles thresholded using the 200
  largest $L_1$ norms, along with the histogram and boxplot.}
\vspace{-.20in}
\label{fig:1stQuadrant_facebook}
\end{figure}

To get more information about the dependence structure,
we construct a diamond plot using
 thresholding corresponding to the 200 largest $L_1$ norms of (out,in).
The scatter plot in Figure \ref{fig:fbwallscatter} is less clear
than for simulated data and shows
points  dispersed from the main cluster about the diagonal
and so the estimates of the
support of $\btheta$ in the diamond plot
are not as evident
as in Figure \ref{fig:1stQuadrant_facebook}. 

{We estimate the support interval } using  the interquartile range and obtain $[.4479,
.5305]=[\theta_l, \theta_u]$. This corresponds to slopes
$(a_l,a_u)=(\theta_u^{-1}-1,\theta_l^{-1}-1)=(.885,1.23).$ We also
include a boxplot of the values of $\theta_1$ corresponding to the 200
largest values of the $L_1$ norm of (out,in).

\vspace{.2in}
\begin{wrapfigure}{l}{0.35\textwidth}
\vspace{-.13in}
\centering
\includegraphics[width=.35\textwidth]{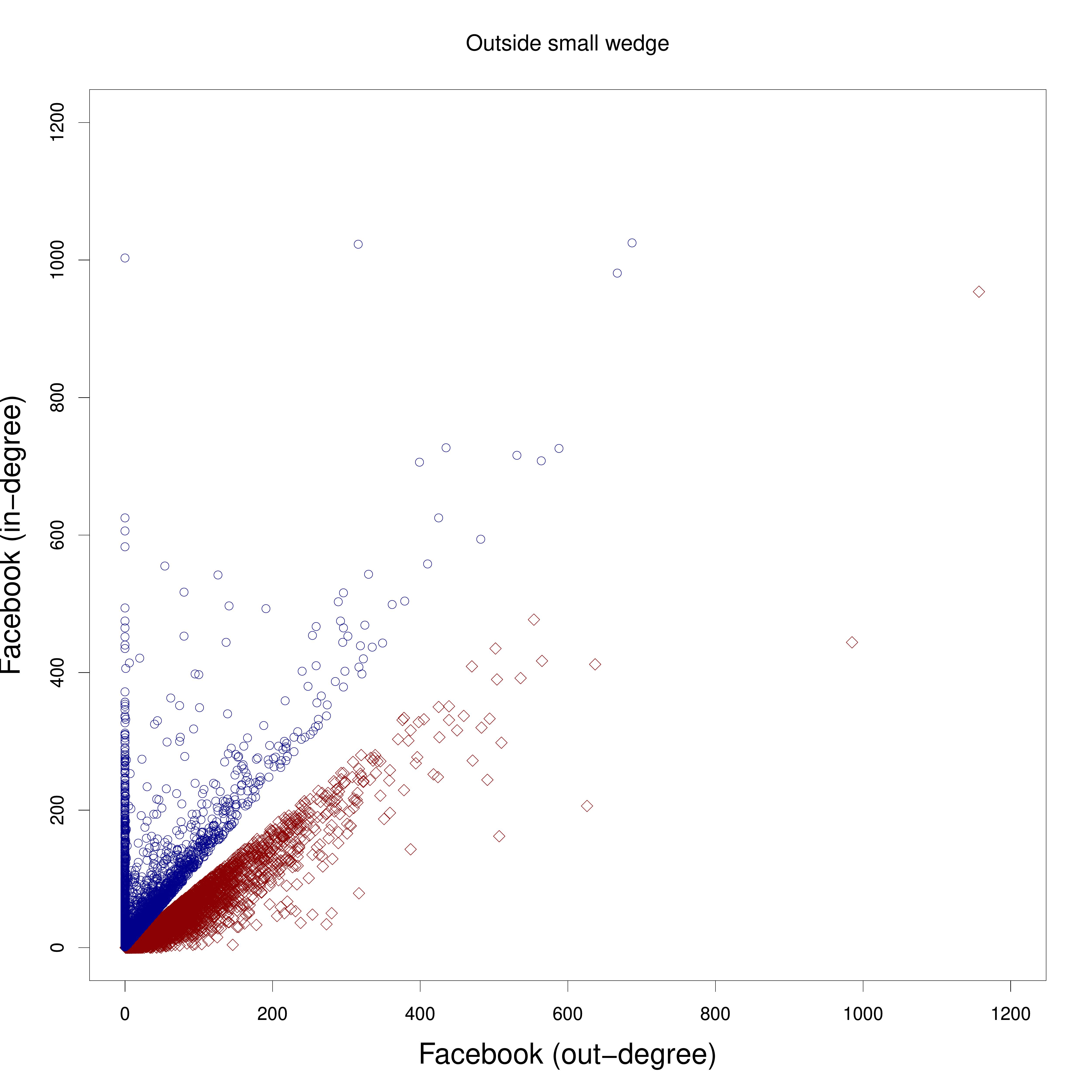}
\vspace{-.23in}
\caption{Points remaining after removal of $[\smallwedge]$.}
\vspace{-.10in}
\label{fig:minusWedge} 
\end{wrapfigure} 

Having determined $[\smallwedge]$, we remove it from the first
quadrant and use the remaining points, illustrated in Figure \ref{fig:minusWedge},
 to seek hidden regular variation.
Preliminary diagnostics use equations \eqref{eq:wedge1st} and
\eqref{eq:wedge2nd} corresponding to points above and below
$[\smallwedge]$ to estimate $\alpha_0$, the index of hidden regular
variation. There are 12,089 points above $[\smallwedge]$ in the region
we refer to as $[>\smallwedge]=\{(x,y): y>(1.23)x>0\}$ and 24,687
below in the region $[<\smallwedge]=\{(x,y): 0<y<(0.885)x\}.$

\begin{figure}[h]
\centering
\includegraphics[width=3.5in]{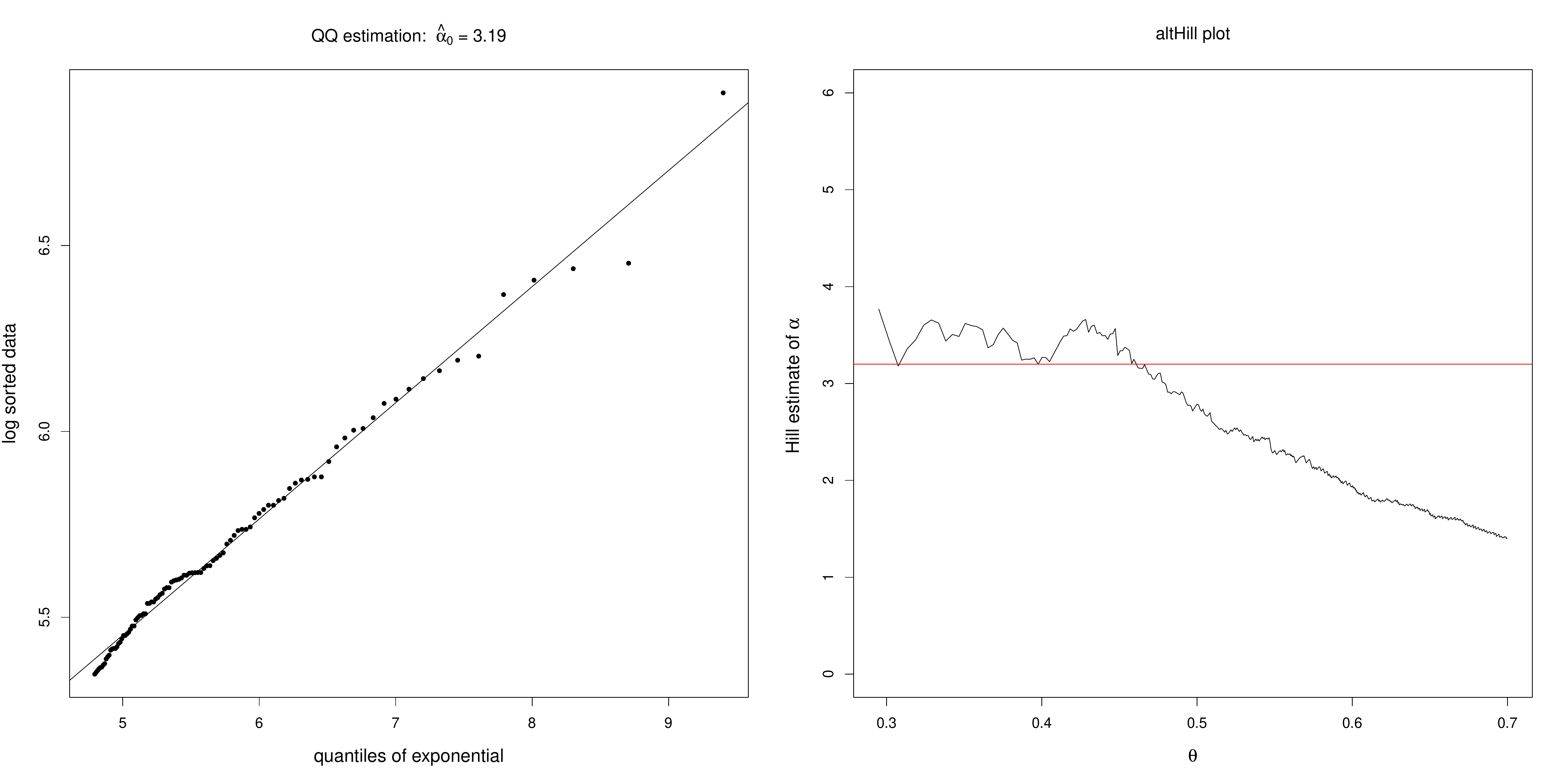}
\caption{Facebook above $[\smallwedge]$: QQ estimate is 3.2 for
  $k=100$ and altHill is 3.3; the red line is at height 3.3.}
\label{fig:facebk0alpha0}
\end{figure}

For the region  $[<\smallwedge]$, a combination of QQ and altHill
plotting gives a stable region for various values of $k$, the number
of upper order statistics, between 100-500 and a value of $\hat
\alpha_0=2.8$, the tail index of $(0.885)Z_1-Z_2$ from
\eqref{eq:wedge2nd}, which is not measurably different from $\hat
\alpha =2.8$ 
found for the marginal distributions of (out,in). This raises doubts
about the presence of HRV in the region $[<\smallwedge]$. For the
region  $[>\smallwedge]$ we use as data $(Z_2-(1.23)Z_1)_+$  and estimate $\hat
\alpha_0\approx 3.2$. The QQ-estimate is 3.2 and altHill gives about
3.3; both estimates are greater than $\alpha=2.8$ so there is evidence
of the existence of HRV in the region above $[\smallwedge]$. The QQ
and altHill plots for the region $[>\smallwedge]$ are given in
Figure \ref{fig:facebk0alpha0}. The evidence for HRV is further
strengthened by excellent Hillish plots described in Section
\ref{subsec:Hillish} applied to the generalized polar coordinates 
$(Z_2-(1.23)Z_1, Z_2/Z_1)$ as descibed in
\eqref{eq:>smallwedge}. 
Both Hillish plots in Figure \ref{fig:fb_gpolar} hug the horizontal
line at height 1.

\begin{figure}[h]
\centering
\includegraphics[height=2.5in]{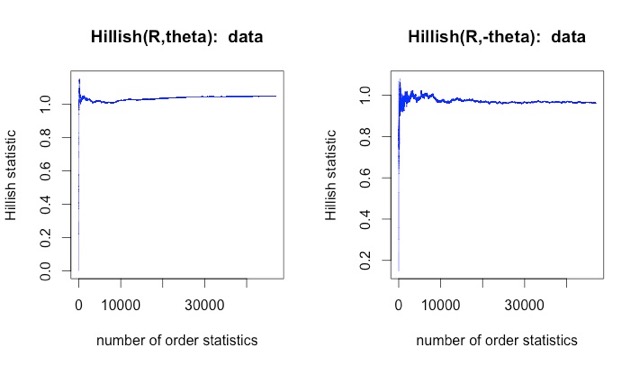}
\caption{Hillish plots for generalized polar coordinates of points in $[>\smallwedge]$.}
\label{fig:fb_gpolar}
\end{figure}

The diamond plot in Figure \ref{fig:1stQuadrant_facebook} shows the presence of points on the
  line $x_1+x_2=1$ with small values of $\theta$, corresponding to two
  dimensional points in $[<\smallwedge]$. Previously the 
  estimation of the range of the angular measure of the primary
  regular variation discounted these
  points. However, the estimation of the tail index of the distance to
  $[\smallwedge]$ being $2.8$, the same as the marginal distribution
  indices of (out,in),  suggests an alternate model which lumps
  together $[\smallwedge]\cup[<\smallwedge]$ as the region of
  concentration for the limit measure $\nu(\cdot)$ of the primary
  regular variation in \eqref{eq:RegVarMeas}. So our alternate model
  is regular variation on $\R_+^2 \setminus \{\bzero\}$ with index
  $2.8$ and limit measure which concentrates on $\{\bx \in
  \R_+^2\setminus \{\bzero\}: x_2/x_1<1.23\}$ and hidden regular variation
  on $\R_+^2\setminus ([\smallwedge]\cup [<\smallwedge])$ with index 3.3.

\subsection{Exxon and Chevron returns.}\label{subsec:oil}

\begin{figure}[h]
\centering
\includegraphics[width=3.in]{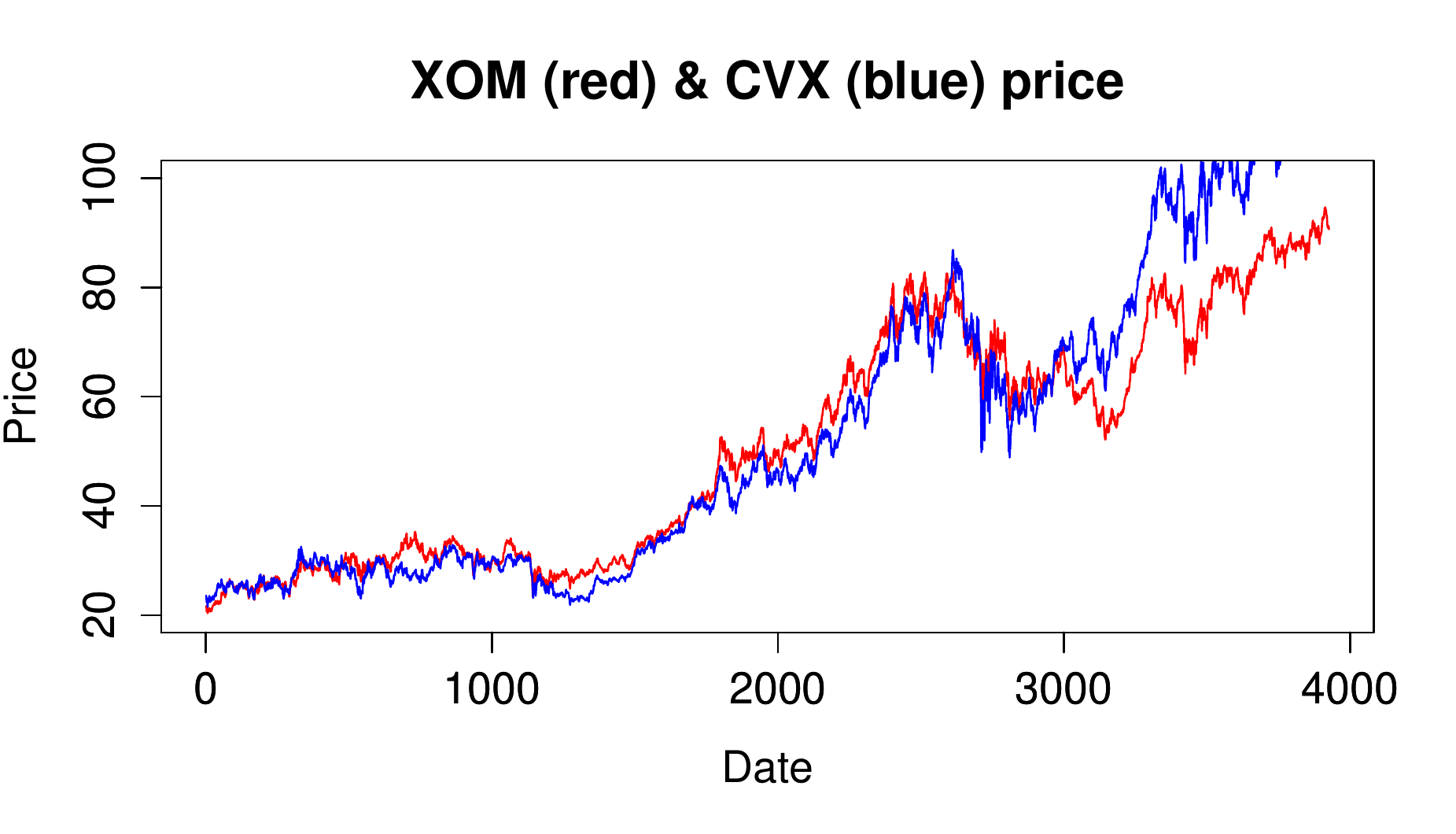}
\includegraphics[width=3.in]{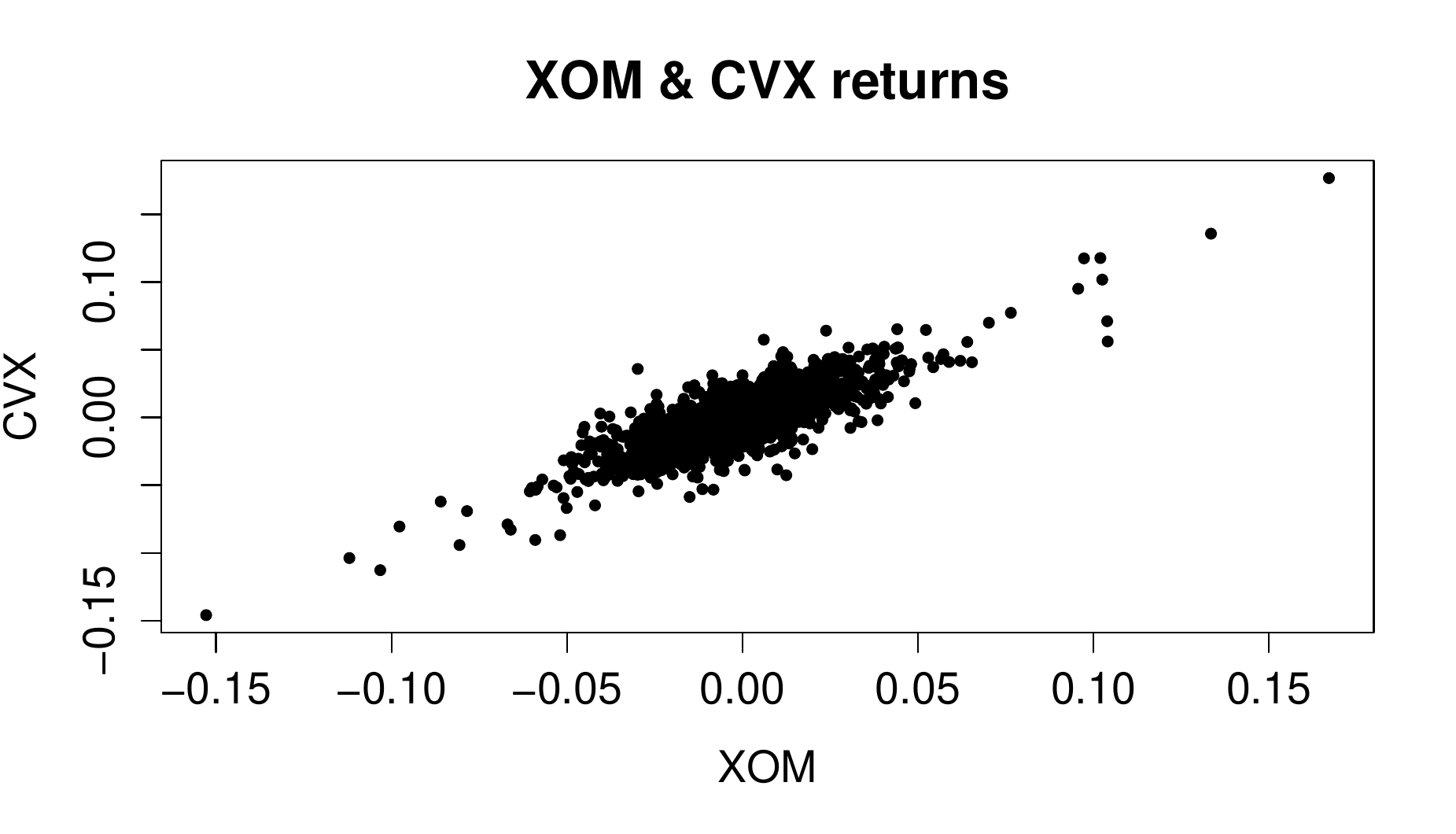}
\vspace{-.2in}
\caption{Stock prices and scatterplot of Chevron and Exxon returns.}
\label{fig:scatter}
\end{figure}
For this example of financial returns, the state space is
$\R^2\setminus \{\bzero\}$ and for the HRV property we could try deleting
$[\smallwedge_+]\subset \R_+^2$ in the first quadrant and
$[\smallwedge_-] \subset (-\infty,0)^2$ in the third quadrant. 
For illustration, we concentrate on deleting only a wedge from the first
quadrant and seeking HRV with points above the upper boundary of 
$[\smallwedge_+]$.
This is done partly because there is no guarantee that HRV will hold
globally.
 The data consists of closing daily prices of Exxon (XOM) and Chevron (CVX)
from January 2, 1998 to August 9, 2013.
For each variable we  calculate
daily returns
for each company 
called (exxonr, chevronr). The length of the return vector is 3925.
One expects strong
dependence from two big companies engaged in similar economic activities and
this is shown in the raw scatter plot of the variables in Figure \ref{fig:scatter}.

The four tails of the variables ($\pm$exxonr, $\pm$chevronr) are
quite similar. Based on analyses (not shown) using the QQ
estimator, Hill and altHill plots,
(eg. \cite{kratz:resnick:1996}, \cite[p. 101, 366]{resnickbook:2007})
we estimate 
\begin{wrapfigure}{l}{.45\textwidth}
\includegraphics[width=3in]{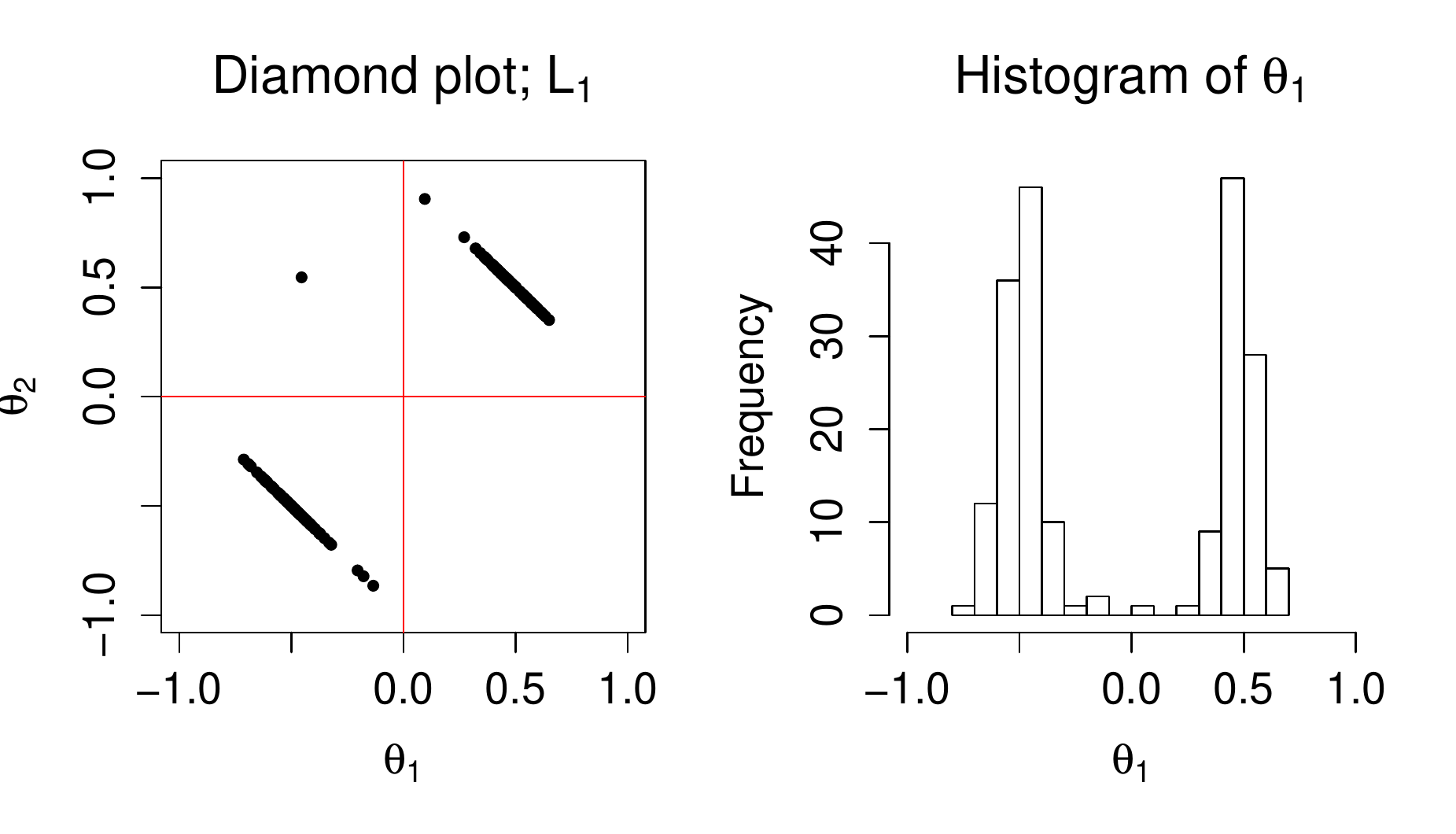}
\caption{Diamond plot for 200 largest values  under
  $L_1$ norm for (exxonr,chevronr) with histogram.} 
\label{fig:oil_diamond}
\end{wrapfigure}
the marginal tail indices $\alpha =2.7$ in all four cases. Since
the tails are estimated to have the same $\alpha$, we did 
not  attempt  to standardize the variables to $\alpha=1$ as is often done by
either the power method or the ranks transform.

To understand the dependence structure of the variables
(exxonr,chevronr), we {make} a diamond plot of the data.
We do the mapping after thresholding the data at various values
 determined by $k$, the number
of order statistics of the norms 
$|x_1|+|x_2|$. This is the two-tail  empirical equivalent to
\eqref{eq:regVarEPolar} using the $L_1$ norm. After experimenting with
thresholds, we settled on $k=200$ which in the first quadrant produced
a range of $\theta_1=x_1/(x_1+x_2)$ equal to $(.095,.649).$ {We
finalized our estimate of  the support of the limit angular measure},
by using the 10\% and 90\% quantiles of
the values of $\theta_1$  as
$(.393,.589)$. This corresponds to slope estimates for $[\smallwedge]$
of $(\hat a_l,\hat a_u)= (.698,1.545)$. 
The strong asymptotic dependence
  among the marginals is evident 
from the diamond plot and histogram of $\theta_1$ in Figure
\eqref{fig:oil_diamond}.
There is little evidence that a large positive
change in one variable is accompanied by a large negative change in the
other as shown by the lack of points in the second and fourth
quadrants in Figure \ref{fig:oil_diamond}. 
There is no visual evidence
supporting the hypothesis of full asymptotic dependence.

\begin{wrapfigure}{r}{.5\textwidth}
\centering
\includegraphics[width=.45\textwidth]{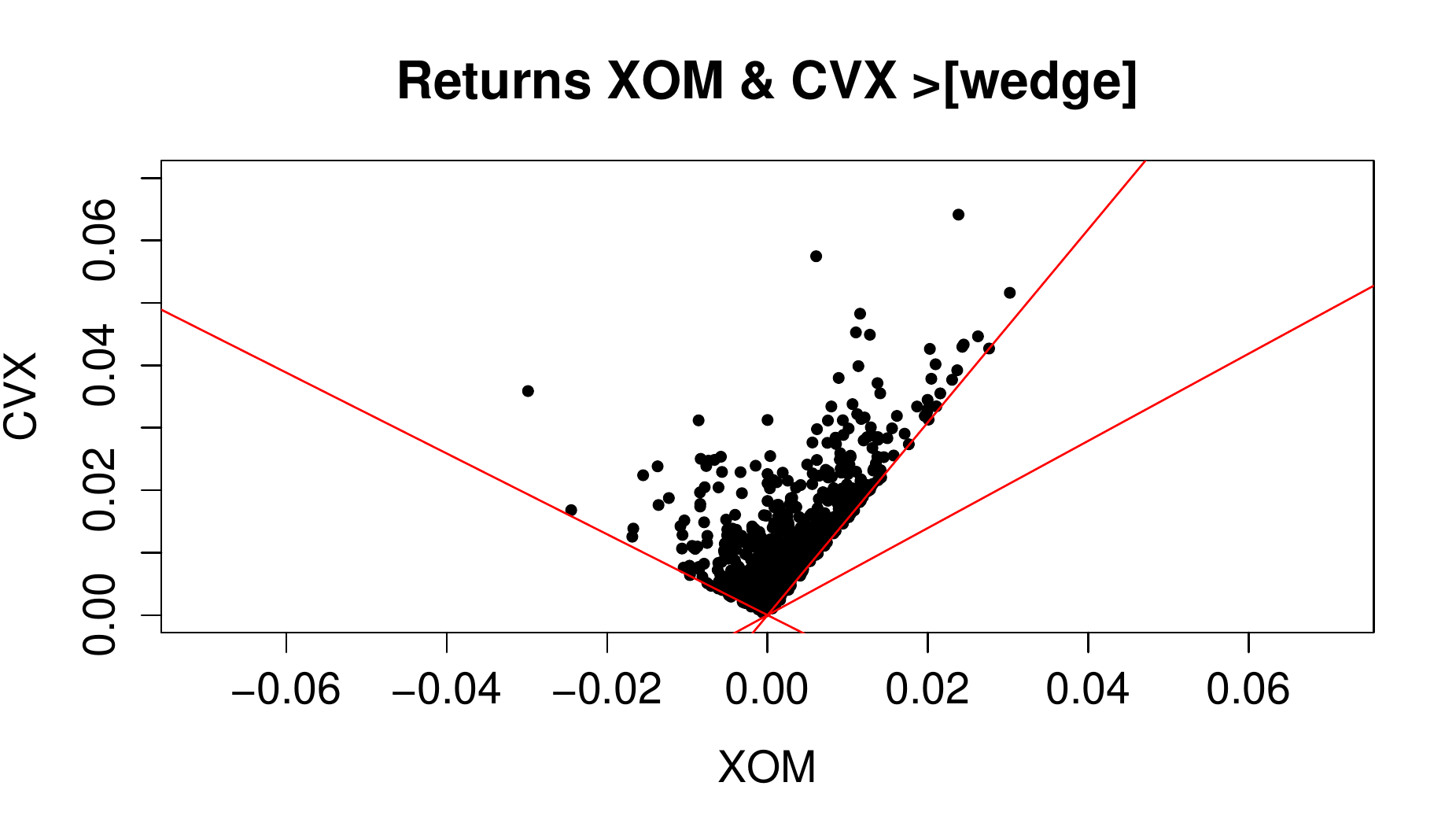}
\caption{Points satisfying three conditions allowing computation of
  gpolar coordinates. Scanning counterclockwise, the first two rays
  define $[\smallwedge_+]$ and the third ray is perpendicular to the
  upper boundary of $[\smallwedge_+]$.}
\label{fig:niceOilypts}
\end{wrapfigure}
Remark \ref{rem:genius} suggests verifying 
the necessary condition for HRV on $\R^2\setminus
[\smallwedge_+]$ by computing the tail index of what is essentially the distance
of a point to $[\smallwedge_+]$.  We seek evidence 
of regular variation on $\R^2\setminus 
[\smallwedge_+]$ by  using points $\bx$ of the
return sample that satisfy,
\begin{enumerate}
\item $x_2>0$ (points above the horizontal axis);
\item $x_2-(1.545)x_1>0$ (points in the first or second quadrant above the
  ray $x_2=1.545x_1,\,x_1>0$)
\item $x_1+1.545x_2>0$ (points in the first or second quadrant
in the region bounded by the ray $x_2=1.545x_1,\,x_1>0$ and the ray
perpendicular to this ray emanating from the origin into the third
quadrant).
Points to the left of this perpendicular would be closer to
$[\smallwedge_-]$ rather than $[\smallwedge_+]$ and are excluded.
\end{enumerate}
There are 706 points satisfying the three conditions; we call these
points {\it oilReturns2013Gr\/}.  These are plotted in Figure
\ref{fig:niceOilypts}. The angle between the two rays of biggest slope is 90 degrees.

We estimate the tail index $\alpha_0$ of the distance of points in {\it
  oilReturns2013Gr\/} to the boundary of $[\smallwedge_+]$ 
to be greater than $\alpha =2.7$ using altHill, Hill and QQ
plots. This corresponds to estimating the tail index of $Z_2-a_uZ_1$
as in \eqref{eq:wedge1st}. The plots are
given next.
\begin{figure}[h]
\includegraphics[width=2.in,height=1.5in]{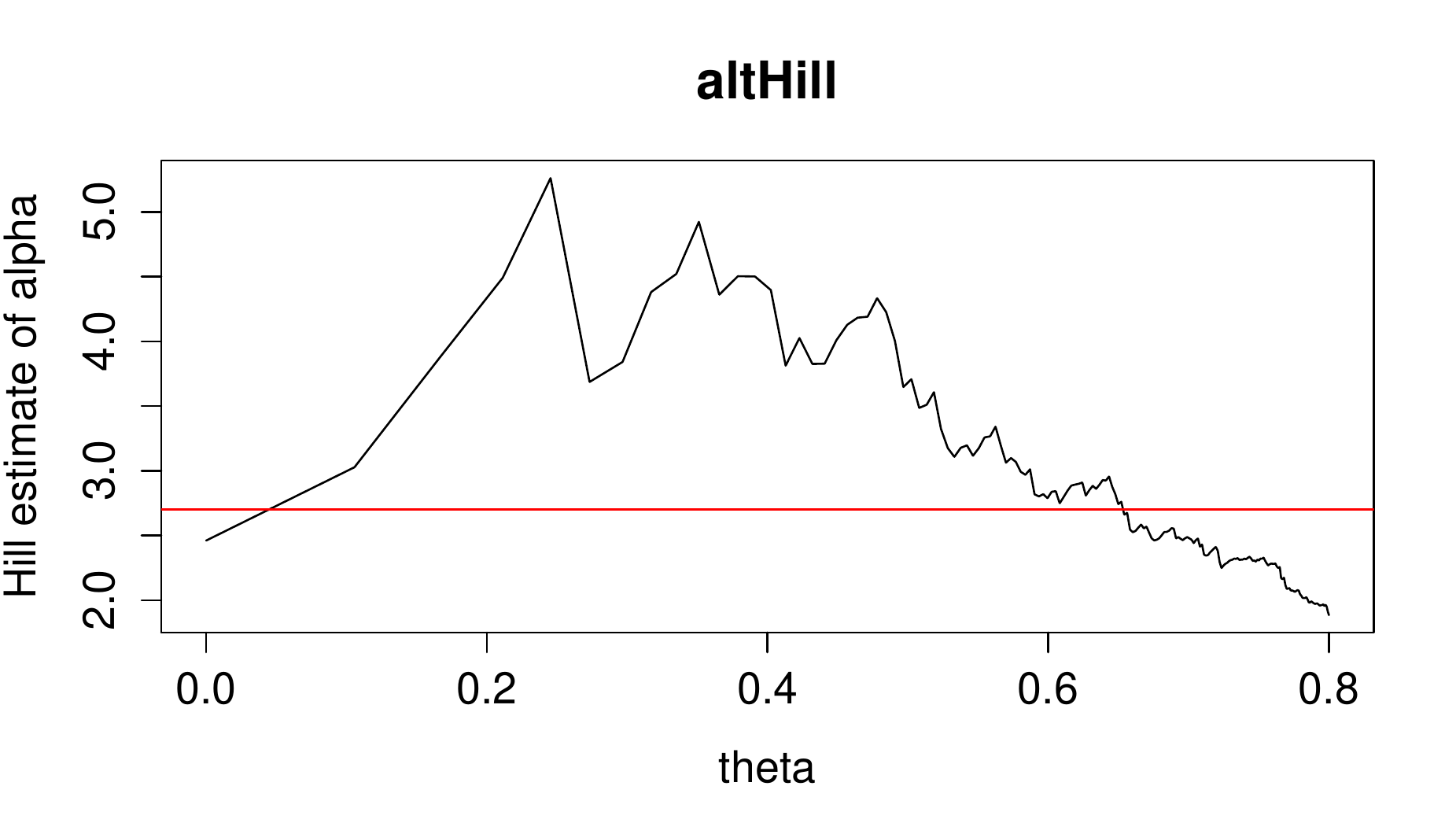}
\includegraphics[width=2.in,height=1.5in]{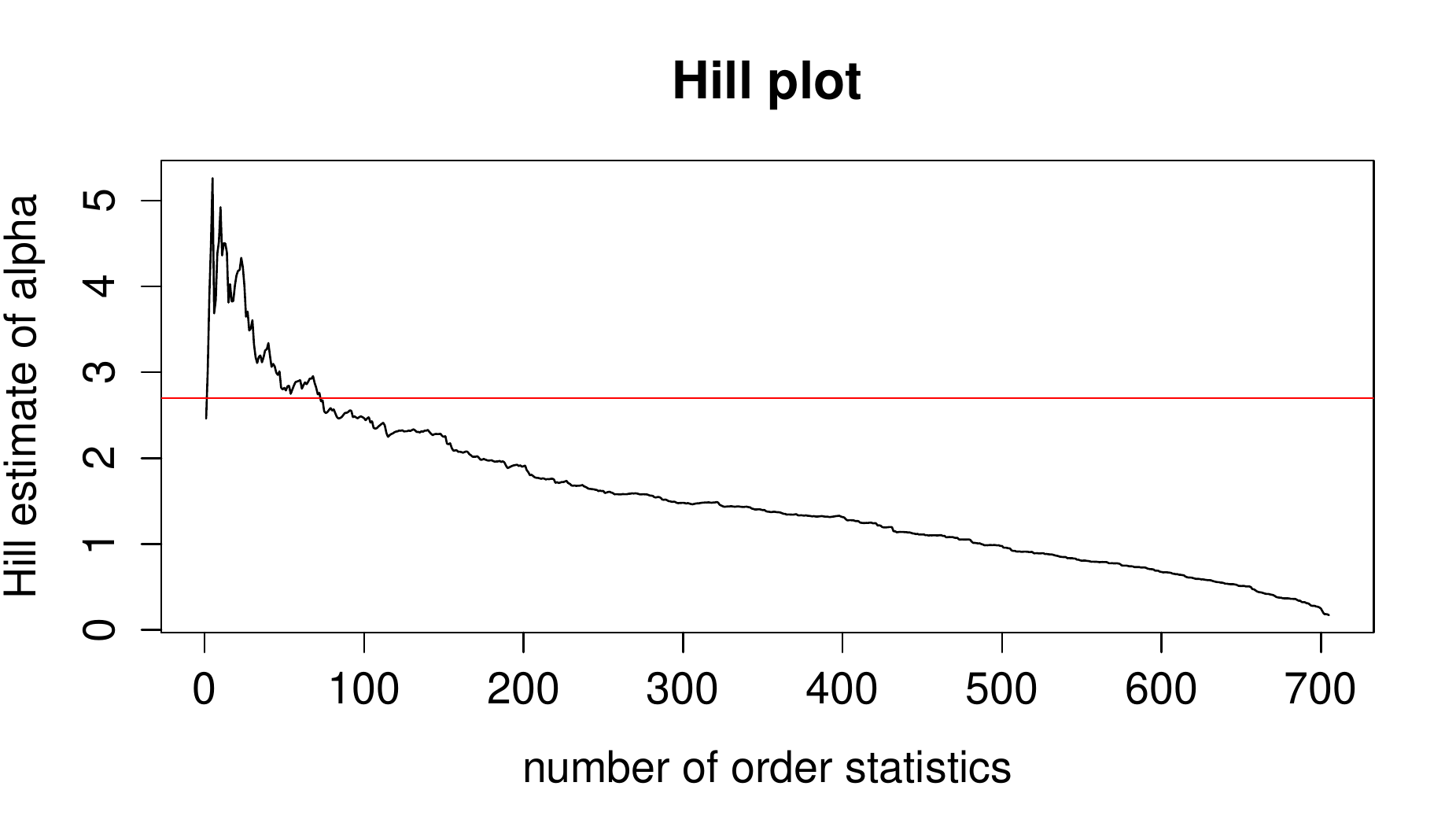}
\includegraphics[width=2.in,height=1.5in]{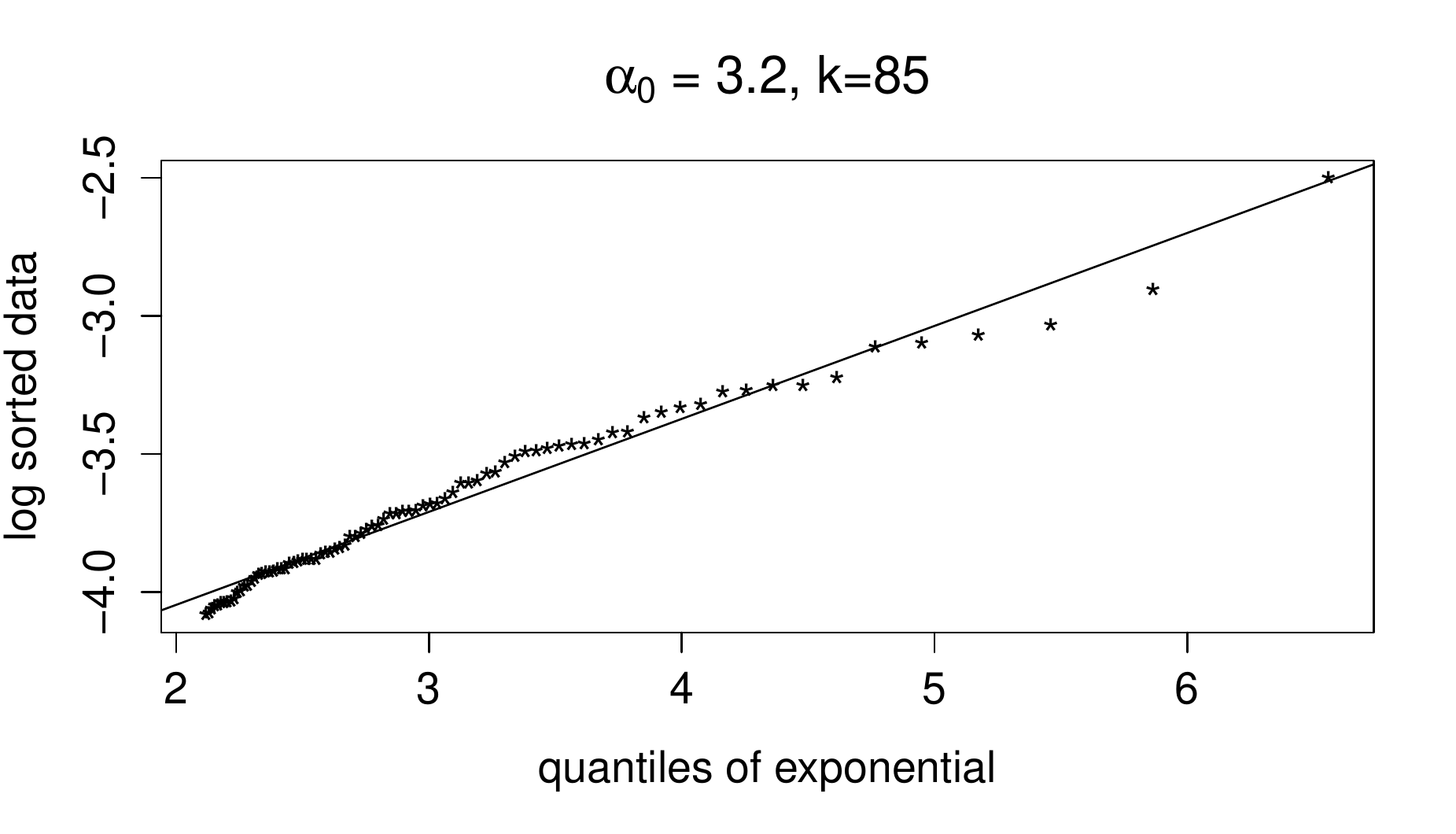}
\caption{AltHill, Hill and QQ plots to estimate $\alpha_0$. The red
  horizontal lines are drawn at height $\alpha=2.7$.}
\label{fig:csReturnsGr}
\end{figure}

\begin{figure}[h]
\vspace{-.34in}
\centering
\includegraphics[width=.6\textwidth]{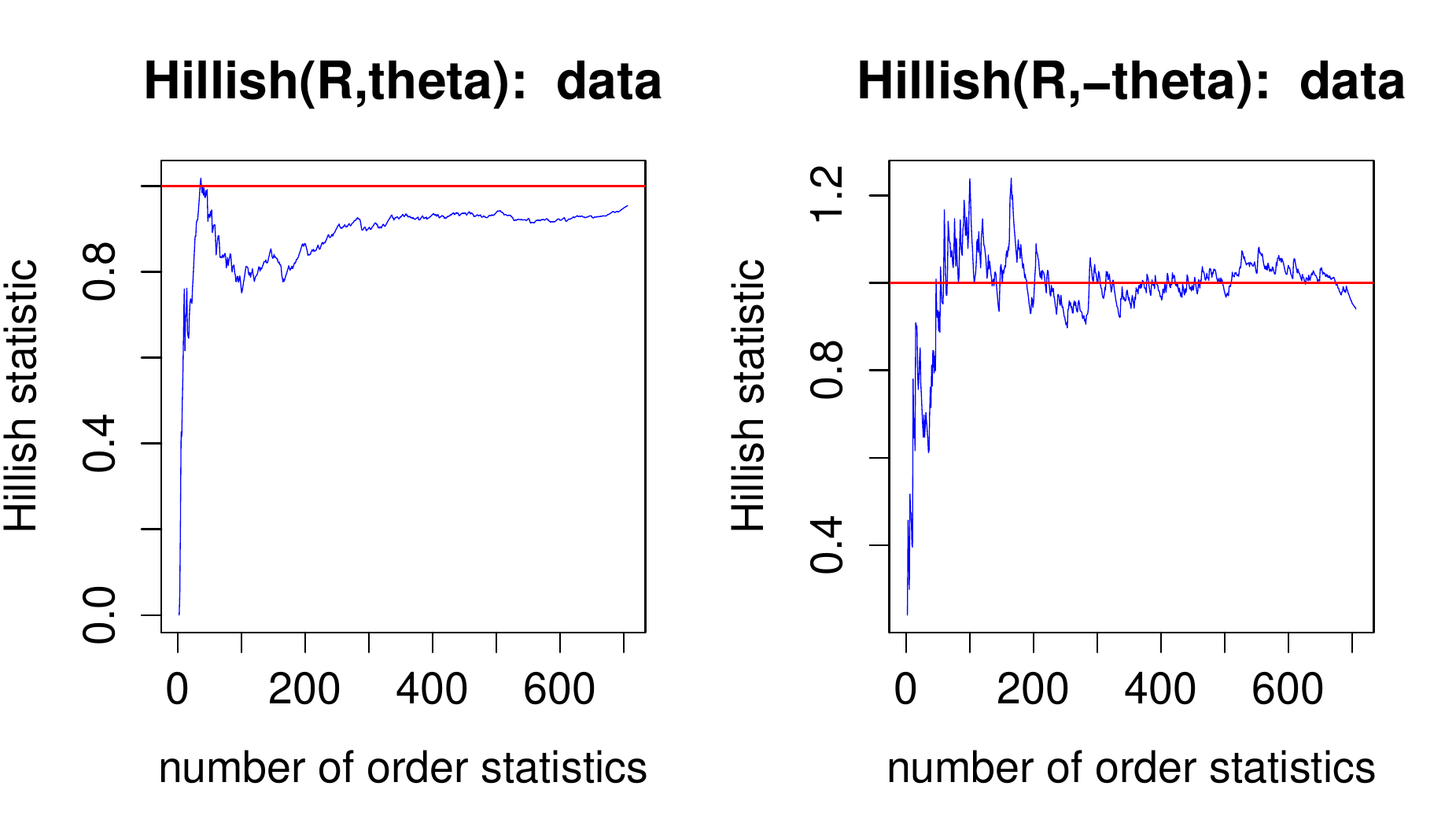}
\caption{Hillish plots for generalized polar coordinates of points in
  {\it oilReturns2013Gr\/}.}
\label{fig:HillisPosRet}
\end{figure}
More evidence for existence of HRV in the region corresponding to {\it
  oilReturns2013Gr\/} is provided by converting the data points in
this region using the generalized polar coordinates
 suggested by \eqref{eq:>smallwedge}. This produces the Hillish
plots given in Figure \ref{fig:HillisPosRet}. Both Figures
\ref{fig:csReturnsGr} and \ref{fig:HillisPosRet} are encouraging.

\section{Conclusions}\label{sec:conc}
Whenever the limit measure of multivariate regular variation
concentrates on a cone smaller than the full state space, there is the
potential for seeking hidden regular variation. This idea has been
most often applied to the case of asymptotic independence where the
limit measure concentrates on the axes. Here we have shown the idea is
also applicable when the limit measure concentrates on the diagonal or
a narrow cone such as $[\smallwedge]$. 

Without hidden regular variation, asymptotic independence causes
analysts to miss risk 
contagion. Analogously, when the limit measure concentrates on the diagonal,
analysis would estimate the probability of a risk region $\{(x,y):
y-x>4\}$ to be zero when in fact, hidden regular variation would yield
a small but non-zero probability. {Our data analyses show potential for  such estimation in strongly dependent data with heavy-tailed marginal distributions.}

Without doubt, much work remains to be done on implementation. Both our
network data which is node based and our returns data is nothing like
independent replicated data. {Also, our methods for estimating the
support of the angular measure $S(\cdot)$ are primitive at
best. Higher dimensional examples present increased visualization and
estimation difficulties. }
None-the-less, we believe the worked out examples are
useful and illustrate practical cases. Other examples exist and in
particular we have analyzed Microsoft vs Dell returns with results
similar to those found in Section \ref{subsec:oil}. 

\section{Acknowledgements}
We acknowledge with thanks the contribution in fall 2014 of Amy Willis
who skilfully analyzed many financial data sets seeking 
examples of asymptotic full and strong asymptotic dependence. We are also
grateful to Paul Embrechts who read an early draft and had many useful
and encouraging comments.
Two referees made many helpful and insightful comments. 

B. Das was supported by  MOE-2013-T2-1-158 and IDG31300110. B. Das also acknowledges
  hospitality from Cornell University during visits in June 2015 and January 2016. S. Resnick was supported by Army MURI grant
  W911NF-12-1-0385 to Cornell University.

\bibliographystyle{plain}


\def\cprime{$'$}

\end{document}